\documentclass{amsart}
\usepackage[utf8]{inputenc}
\usepackage{tikz-cd}
\usepackage{amsmath}
\usepackage{verbatim}
\usepackage{amssymb}
\usepackage{amsthm}
\usepackage{hyperref}
\usepackage{multirow}
\usepackage[bb=dsserif]{mathalpha}

\newtheorem*{theorem*}{Theorem}

\usepackage{ulem}

\usetikzlibrary{calc}
\usetikzlibrary{decorations.pathmorphing}
\tikzset{curve/.style={settings={#1},to path={(\tikztostart)
    .. controls ($(\tikztostart)!\pv{pos}!(\tikztotarget)!\pv{height}!270:(\tikztotarget)$)
    and ($(\tikztostart)!1-\pv{pos}!(\tikztotarget)!\pv{height}!270:(\tikztotarget)$)
    .. (\tikztotarget)\tikztonodes}},
    settings/.code={\tikzset{quiver/.cd,#1}
        \def\pv##1{\pgfkeysvalueof{/tikz/quiver/##1}}},
    quiver/.cd,pos/.initial=0.35,height/.initial=0}

\tikzset{tail reversed/.code={\pgfsetarrowsstart{tikzcd to}}}
\tikzset{2tail/.code={\pgfsetarrowsstart{Implies[reversed]}}}
\tikzset{2tail reversed/.code={\pgfsetarrowsstart{Implies}}}
\tikzset{no body/.style={/tikz/dash pattern=on 0 off 1mm}}

\DeclareMathOperator{\im}{im}

\DeclareMathOperator{\field}{k}
\DeclareMathOperator{\op}{op}
\DeclareMathOperator{\rad}{rad}

\DeclareMathOperator{\modu}{mod}
\DeclareMathOperator{\Ext}{Ext}

\DeclareMathOperator{\End}{End}
\DeclareMathOperator{\Hom}{Hom}

\DeclareMathOperator{\Sim}{Sim}

\DeclareMathOperator{\id}{id}

\DeclareMathOperator{\Res}{Res}

\DeclareMathOperator{\Inn}{Inn}
\DeclareMathOperator{\ess}{ess}

\DeclareMathOperator{\filt}{F}
\newtheorem{theorem}{Theorem}[section]
\newtheorem{definition}[theorem]{Definition}
\newtheorem{example}[theorem]{Example}
\newtheorem{corollary}[theorem]{Corollary}
\newtheorem{lemma}[theorem]{Lemma}
\newtheorem{remark}[theorem]{Remark}
\newtheorem{proposition}[theorem]{Proposition}
\title{Exact Borel Subalgebras of Quasi-Hereditary Monomial Algebras}
\author{Anna Rodriguez Rasmussen}

\DeclareMathOperator{\qh}{\mathbb{qh}}

\begin{document}
\bibliographystyle{plain}
\begin{abstract}
    Green and Schroll give an easy criterion for a monomial algebra $A$ to be quasi-hereditary with respect to some partial order $\leq_A$. A natural follow-up question is under which conditions a monomial quasi-hereditary algebra $(A, \leq_A)$ admits an exact Borel subalgebra in the sense of König. In this article, we show that it always admits a Reedy decomposition consisting of an exact Borel subalgebra $B$, which has a basis given by paths, and a dual subalgebra. Moreover, we give an explicit description of $B$ and show that it is the unique exact Borel subalgebra of $A$ with a basis given by paths. Additionally, we give a criterion for when $B$ is regular, using a criterion by Conde.
\end{abstract}
\maketitle
\tableofcontents
\section{Introduction}
In \cite{CPS}, Cline, Parshall and Scott defined the notion of a quasi-hereditary algebra, which is a pair $(A, \leq_A)$ consisting of a finite-dimensional algebra $A$ and a partial order $\leq_A$ on the set $\Sim(A)$ of isomorphism classes of simple $A$-modules, fulfilling certain properties. Quasi-hereditary algebras are a generalisation of hereditary algebras, that is, algebras such that any submodule of a projective module is projective, in that for a hereditary algebra $A$, $(A, \leq_A)$ is quasi-hereditary for any total order $\leq_A$ on $\Sim(A)$. However, quasi-hereditary algebras encompass many more algebras, such as, for example Schur algebras, algebras of global dimension two and algebras underlying blocks of BGG category $\mathcal{O}$.\\ 
Given a quasi-hereditary algebra $(A, \leq_A)$ with simple modules $L_1^A, \dots, L_n^A$, one often considers the corresponding standard modules $\Delta^A_1, \dots, \Delta_n^A$ and the category $\filt(\Delta^A)$ of standardly filtered modules. In the case of blocks of category $\mathcal{O}$, the standard modules are Verma modules and $\filt(\Delta^A)$ is the category of modules admitting a Verma flag.
By \cite[Theorem 2]{DlabRingel},  the category $\filt(\Delta^A)$ determines the quasi-hereditary algebra $(A, \leq_A)$ up to Morita equivalence.\\
Inspired by the notion of a Borel subalgebra of a Lie algebra, König defined in \cite{Koenig} the notion of an exact Borel subalgebra of a quasi-hereditary algebra. Such a subalgebra, if it exists, carries a lot of information about the category $\filt(\Delta^A)$ of standardly filtered $A$-modules, and can thus be used to investigate the quasi-hereditary structure of $(A, \leq_A)$.
Unfortunately, a quasi-hereditary algebra $(A, \leq_A)$ may fail to admit an exact Borel subalgebra \cite[Example 2.3]{Koenig}. Nevertheless, a landmark result by König, Külshammer and Ovsienko \cite{KKO} shows that any quasi-hereditary algebra $(A, \leq_A)$ is Morita equivalent to a quasi-hereditary algebra $(R, \leq_R)$ which does admit an exact Borel subalgebra $B\subseteq R$ which is moreover a so-called regular exact Borel subalgebra. The proof is constructive; however, the construction involves calculating the A-infinity structure of the Ext-algebra of the standard modules, which is in general very difficult.\\
Dually to exact Borel subalgebras, one may study $\Delta$-subalgebras, which correspond to exact Borel subalgebras in the opposite algebra. Certain quasi-hereditary algebras $A$ admit both an exact Borel subalgebra $B$ and a $\Delta$-subalgebra $C$, as well as a decomposition $A\cong C\otimes_L B$, where $L$ is a maximal semisimple subalgebra of $A$. Such algebras were defined in \cite{Koenig2} and recently linked to Reedy categories in \cite{DaleziosStovicek} and \cite{KoenigDaleziosConde}. \\
In this article, we consider the special case where $A$ is a basic monomial algebra. Monomial algebras are a class of algebras more general than path algebras, which nevertheless are much easier to work with than general finite-dimensional algebras. Notably, it was shown in \cite{finitistic} that monomial algebras fulfill the finitistic dimension conjecture.  In a monomial algebra $A=\field Q/I$ it is still possible to work explicitly with paths in the corresponding quiver $Q$. Hence, one may use combinatorial tools to study possible exact Borel subalgebras of $(A, \leq_A)$. Some previous results in this direction for dual extension algebras, Ringel duals of dual extension algebras, and quivers of Dynkin type $\textup{A}$ can be found in \cite{Xi, Markus2, Markus}, although the approach we take here is more similar to \cite{schroll}, where exact Borel subalgebras were not studied.\\
The main result of the article is the following:
\begin{theorem*}\ref{thm_borel_hereditary}+\ref{thm_reedy}+\ref{proposition_uniqueness}
    Let $(A, \leq)$ be a monomial quasi-hereditary algebra.\\ Then $(A, \leq)$ admits a Reedy decomposition $A=C\otimes_L B$, where $B$ is an exact Borel subalgebra with a basis given by paths, $L$ is the span of the trivial paths, and $C$ is a $\Delta$-subalgebra with a basis given by paths. Moreover, $B$ is the unique exact Borel subalgebra of $A$ which has a basis given by paths, and $C$ is the unique $\Delta$-subalgebra of $A$ with a basis given by paths.
\end{theorem*}
While exact Borel subalgebras are in general not unique up to isomorphism,
there are several known uniqueness results for regular exact Borel subalgebras, see for example \cite{Conde, Miemietz, uniqueness}. In particular, if $A$ is a quasi-hereditary monomial algebra, then since $A$ is basic it follows from \cite[Theorem 8.4]{uniqueness} that if $A$ admits a regular exact Borel subalgebra, then this is up to inner automorphism the unique exact Borel subalgebra of $A$. However, if $A$ does not admit a regular exact Borel subalgebra, then even in the monomial case it may admit exact Borel subalgebras which are not conjugate to one another, as in Example \ref{counterexample}.\\
Given that regularity plays such a central role in the study of exact Borel subalgebras, we furthermore investigate under which conditions a quasi-hereditary monomial algebra $A$ admits a regular exact Borel subalgebra.  Under the assumption that $\field$ is algebraically closed, we obtain a rather technical characterization in terms of paths and relations of $A$. In the hereditary case, this simplifies to the following:
\begin{theorem*}\ref{proposition_regularity_hereditary}
    Suppose $A=\field Q$ is a hereditary algebra, and $\leq$ is a total order on the vertices of $Q$. Then $(A, \leq)$ admits a regular exact Borel subalgebra if and only if the following two conditions hold:
    \begin{enumerate}
        \item For all paths $p:i\rightarrow k$ and $q:j\rightarrow k$ with $\max(p)=k>i$ and $j>k$ there is a path $r:j\rightarrow i$ such that $q=pr$.
        \item For any path $q:j\rightarrow i$ with $j>i$, there is at most one path $p$ starting in $i$ such that $p:i\rightarrow k$ for some $j>k>i$ and $k'<i$ for every other vertex $k'$ that $p$ passes through.
    \end{enumerate}
\end{theorem*}
As an example, we consider path algebras of quivers $Q(n_a, n_b, n_c)$ of the form
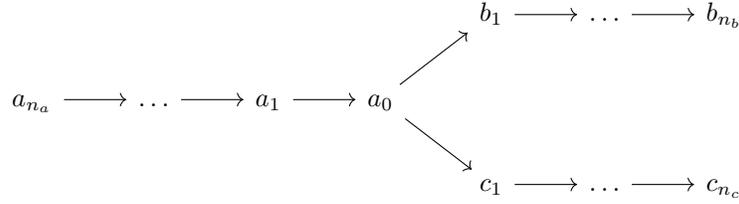
\[\begin{tikzcd}[ampersand replacement=\&]
	\&\&\&\& {b_1} \& \dots \& {b_{n_b}} \\
	{a_{n_a}} \& \dots \& {a_1} \& {a_0} \\
	\&\&\&\& {c_1} \& \dots \& {c_{n_c}}
	\arrow[from=1-5, to=1-6]
	\arrow[from=1-6, to=1-7]
	\arrow[from=2-1, to=2-2]
	\arrow[from=2-2, to=2-3]
	\arrow[from=2-3, to=2-4]
	\arrow[from=2-4, to=1-5]
	\arrow[from=2-4, to=3-5]
	\arrow[from=3-5, to=3-6]
	\arrow[from=3-6, to=3-7]
\end{tikzcd}\]
for $n_a, n_b, n_c\geq 0$ and investigate when they admit a regular exact Borel subalgebra. In Theorem \ref{thm_counting}, we give a formula counting the number of quasi-hereditary structures on path algebras of quivers of the form $Q(n_a, n_b, n_c)$ which admit a regular exact Borel subalgebra.  For quivers of Dynkin type D and E, we contrast this with the overall number of quasi-hereditary structures, which was calculated in \cite{combinatorics}.  
For example, for the path algebra of the quiver $Q(n, 1, 1)$, we obtain the following:
\begin{theorem*}(This is a special case of Theorem \ref{thm_counting})\\
    Out of the $2C_{n+3}-3C_{n+2}$ distinct quasi-hereditary structures on the path algebra of the quiver $Q(n, 1, 1)$, $4C_{n+2}-4C_{n+1}$ admit a regular exact subalgebra. Here, $C_k$ denotes the $k$-th Catalan number. In the limit, the share of quasi-hereditary structures admitting a regular exact Borel subalgebra tends to $3/5$.
\end{theorem*}
The structure of the article is as follows:\\
In Section \ref{section2} we fix our notation, and recall basic definitions and statements regarding quasi-hereditary algebras, exact Borel subalgebras and basic monomial algebras.\\
In Section \ref{section3} we prove that every basic monomial quasi-hereditary algebra  $A$ admits a Reedy decomposition. We explicitly describe the corresponding exact Borel subalgebra $B_{\min}$, and give a decomposition of $A$ as a right $B_{\min}$-module. By \cite{Conde2}, any exact Borel subalgebra is normal. Here, we give an explicit splitting of the inclusion $B_{\min}\rightarrow A$, exhibiting the normality of $B_{\min}$.\\
In Section \ref{section4}, we show that  $B_{\min}$ is the unique exact Borel subalgebra of $A$ which has a basis given by paths. We give an example showing that $A$ may admit other exact Borel subalgebras which do not have such a basis and are not conjugate to  $B_{\min}$.\\
In Section \ref{section5} we establish a criterion for regularity of $B_{\min}$ in terms of paths in $A$, under the assumption that $\field$ is algebraically closed. Since this criterion is nevertheless somewhat complicated, we give a simplified version in case that $A$ is hereditary.\\
In Section \ref{section6} we consider idempotent subalgebras and quotient algebras of a given monomial algebra $A$, and investigate in how far our constructions are compatible with these operations.\\
Finally, in Section \ref{section7} we study path algebras of quivers of the form $Q(n_a, n_b, n_c)$ and $Q(n_a, n_b, n_c)^{\op}$. These quivers include the ADE Dynkin diagrams when endowed with an orientation such that every vertex has at most one ingoing and at most one outgoing arrow. In view of Gabriel's theorem \cite{Gabriel1972} and the deconcatenation results in \cite{combinatorics} and \cite{Markus}, such ADE quivers are precisely the representation-finite path algebras of interest from the point of view of quasi-hereditary algebras, as the other cases can be reduced, by \cite{combinatorics} and \cite{Markus}, to a combination of the cases above and path algebras of quivers of type A, which were already studied in  \cite{combinatorics} and \cite{Markus}.
In Corollary \ref{corollary_special_quiver_regular}, give a simple criterion for when $\field Q(n_a, n_b, n_c)$ respectively $\field Q(n_a, n_b, n_c)^{\op}$ admits a regular exact Borel subalgebra. We use this in order to count the number of quasi-hereditary structures on $\field Q(n_a, n_b, n_c)$ and $\field Q(n_a, n_b, n_c)^{\op}$ that admit regular exact Borel subalgebras, and specialize the result to quivers of type $\textup{D}$ and $\textup{E}$ endowed with an orientation such that every vertex has at most one ingoing and one outgoing arrow. 

\section{Notation and Preliminaries}\label{section2}
Throughout, let $\field$ be a field. All $\field$-vector spaces we consider are finite-dimensional, and all tensor products, unless otherwise stated, are over $\field$. For any subset $S$ of a vector space $V$ we denote by $\langle S\rangle $ the $\field$-span of $S$.\\
Let $Q=(V_Q, E_Q)$ be a quiver with vertex set $V_Q=\{1, \dots, n\}$, let $I$ be an admissible monomial ideal in $\field Q$, that is, $I$ is generated by a set of paths of length at least two in $Q$, and let $A=\field Q/I$ be the corresponding basic monomial algebra. 
Denote by $e_i$ the trivial path in $\field Q$ at the vertex $i$, and by $e_i^A=e_i+I$ the corresponding idempotent in $A$.\\
Note that every element $x\in A$ can be written in a unique way as a sum $x=\sum_i p_i+I$ such that $p_i\notin I$ is a path in $Q$ for every $i$. We call $\sum_i p_i$ the standard representative of $x$, and we call $x$ a path in $A$ if its standard representative is a path in $Q$. Note that, in particular, $0$ is not a path, since the standard representative of $0$ is $0$.\\
We call $x$ homogeneous of degree $k$ if its standard representative consists of paths of length $k$. This gives $A$ the structure of a graded algebra. Moreover,  $L:=A_0=\bigoplus_{i\in V_Q}\field e_i^A$ is a maximal semisimple subalgebra of $A$.\\
For a path $p\in e_j(\field Q)e_i$ in $Q$ we write $s(p)=i$, $t(p)=j$, and $p:i\rightarrow j$. Moreover, we denote by $V_p$ the set of all vertices which $p$ passes through, i.e. $k\in V_p$ if and only if there are (possibly trivial) paths $p', p''\in \field Q$ such that $p=p'p''$ and $s(p')=t(p'')=k$. We call $k$ an inner vertex if there are non-trivial paths $p', p''\in \field Q$ such that $p=p'p''$ and $s(p')=t(p'')=k$ and denote by $\Inn(p)$ the set of inner vertices of $p$.
Moreover, we write $\max(p):=\max V_p$ for the maximal vertex, with respect to the natural order, that $p$ passes through.\\
For $1\leq i\leq n$, let $P_i^A:=Ae_i^A$ be the projective, $I_i^A:=\Hom_{\field}(e_iA, \field)$ be the injective and $L_i^A:=P_i^A/\rad(P_i^A)=Ae_i^A/\rad(A)e_i^A$ be the simple corresponding to $i$, and let $\Sim(A)=\{L_i^A, \dots, L_n^A\}$. Moreover, let $\leq_A$ be the partial order on $\Sim(A)$ given by 
\begin{align*}
    L_i^A\leq_A L_j^A:\Leftrightarrow i\leq j.
\end{align*}
For $1\leq i\leq n$ denote by $\Delta_i^A$ the $i$-th standard module of $A$, i.e. the maximal factor module of $P_i$ such that $[\Delta_i^A: L_j]=0$ for all $j>i$, and by $\nabla_i^A$ the $i$-th costandard module of $A$, i.e. the maximal submodule  of $I_i$ such that $[\nabla_i^A: L_j]=0$ for all $j>i$.
Denote by $\filt(\Delta^A)$ the full subcategory of $A$ consisting of all modules which admit a filtration by standard modules.
Recall that $(A, \leq_A)$ is called quasi-hereditary if and only if $A\in \filt(\Delta^A)$ and ${\End_i(\Delta_i^A)\cong \field}$ for all $1\leq i\leq n$.
Quasi-hereditary algebras were first defined by Cline, Parshall and Scott \cite{CPS}; an introduction can be found for example in \cite{DlabRingel}.
By \cite{CPS}, if $(A, \leq_A)$ is quasi-hereditary, then so is $(A^{\op}, \leq_A)$, and $\Delta_i^{A^{\op}}\cong \Hom_{\field}(\nabla_i^A, \field)$ for every $1\leq i\leq n$. In particular, if $(A, \leq_A)$ is quasi-hereditary, then $[\nabla_i^A:L_i^A]=1$
for all $1\leq i\leq n$.\\
In \cite{Koenig}, König defined the concept of an exact Borel subalgebra for a quasi-hereditary algebra, which is inspired by Borel subalgebras of Lie algebras.
\begin{definition}
    A subalgebra $B$ of a quasi-hereditary algebra $(A, \leq_A)$ is called an exact Borel subalgebra if
    \begin{enumerate}
        \item The induction functor  \begin{align*}
        A\otimes_B -:\modu B\rightarrow \modu A
    \end{align*}
    is exact.
    \item There is a bijection $\phi:\Sim(B)\rightarrow\Sim(A)$ such that for all $L\in \Sim(B)$ we have
    \begin{align*}
        A\otimes_B L\cong \Delta(\phi(L)).
    \end{align*}
    \item For all $L\in \Sim(B)$ we have $\End_B(L)\cong \End_A(\phi(L))$.
    \item $(B, \leq_B)$ is directed with respect to the partial order
    \begin{align*}
        L\leq_B L':\Leftrightarrow \phi(L)\leq_A\phi(L'),
    \end{align*}
    i.e. for all $L, L'\in \Sim(B)$
        \begin{align*}
        \Ext^1_B(L,L')\neq (0)\Rightarrow L\leq_B L'\Leftrightarrow \phi(L)\leq_A\phi(L').
    \end{align*}
    \end{enumerate}
\end{definition}
Dually, \cite{Koenig} defines the concept of a $\Delta$-subalgebra, and proves the following theorem:
\begin{theorem}\cite[Theorem B]{Koenig}
    Let $(A, \leq_A)$ be a quasi-hereditary algebra. Then $C\subseteq A$ is a $\Delta$-subaglebra if and only if $C^{\op}\subseteq A^{\op}$ is an exact Borel subaglebra of $(A^{\op}, \leq_A)$.
\end{theorem}
By \cite[Theorem 4.1]{Koenig2}, basic quasi-hereditary algebras $(A, \leq_A)$ which admit both an exact Borel subalgebra $B$ and a $\Delta$-subalgebra $C$, such that $L:=B\cap C$ is a maximal semisimple subalgebra of $A$, have a so-called Reedy decomposition, that is, the multiplication in $A$ gives rise to an isomorphism $C\otimes_L B\rightarrow A$. For an extensive account about Reedy decompositions and the connection to Reedy categories, see  \cite{KoenigDaleziosConde, DaleziosStovicek}.\\
Since we indexed the simple modules of $A$ by $1, \dots, n$, we will usually omit $\phi$ and instead index the simple modules of $B$ by $1, \dots, n$, so that $\phi$ corresponds to the bijection $[L_i^A]\mapsto [L_i^B]$.\\
\begin{remark}
    Note that in most of the literature, the underlying field is assumed to be algebraically closed, so that the condition $\End_B(L)\cong \End_A(\phi(L))$ is omitted, but it can be found e.g. in \cite{Conde2}.
\end{remark}
One often considers exact Borel subalgebras with additional properties, here we give a definition of the ones that we will be interested in in this article.
\begin{definition}\label{definition_regular}(\cite[Definition 3.4]{BKK} and \cite[p. 405]{Koenig})
    \begin{enumerate}
        \item An exact Borel subalgebra $B$ of a quasi-hereditary algebra $(A, \leq_A)$ is called normal if the embedding $\iota:B\rightarrow A$ has a splitting as a right $B$-module homomorphism whose kernel is a right ideal in $A$.
        \item An exact Borel subalgebra $B$ of a quasi-hereditary algebra $(A, \leq_A)$ is called regular if it is normal and for every $n\geq 1$ the maps
        \begin{align*}
            \Ext^n_B(L^B, L^B)\rightarrow   \Ext^n_A(\Delta^A, \Delta^A), [f]\mapsto [\id_A\otimes_B f]
        \end{align*}
        are isomorphisms, where $L^B=\bigoplus_{L_i^B\in \Sim(B)}L_i^B$ and $$\Delta^A:=A\otimes_B L^B\cong \bigoplus_{L_i^A\in \Sim(A)}\Delta(L_i^A).$$
        \item An exact Borel subalgebra $B$ of a quasi-hereditary algebra $(A, \leq_A)$ is called strong if it contains a maximal semisimple subalgebra of $A$.
    \end{enumerate}
\end{definition} 
In \cite{Conde2}, it was shown that any exact Borel subalgebra is automatically normal.
However, in general, a quasi-hereditary algebra may not have an exact Borel subalgebra \cite[Example 2.3]{Koenig}, and even if it does, this subalgebra may not be unique up to isomorphism \cite[Example A.4]{KKO}.\\
Nevertheless, over an algebraically closed field, any quasi-hereditary algebra is Morita equivalent to some quasi-hereditary algebra admitting a basic regular exact Borel subalgebra \cite{KKO}, and basic regular exact Borel subalgebras are unique up to inner automorphism \cite{BKK, uniqueness}.

\section{Existence}\label{section3}
In this section, we will show that any quasi-hereditary monomial algebra $A$ admits an exact Borel subalgebra $B$, which we can describe explicitly. Moreover, we will give a Reedy decomposition of $A$ and use this to give an explicit splitting of the inclusion map exhibiting the normality of $B$. Since $A$ is monomial, we want to work directly with paths in $A$, so that we begin by introducing some basic terminology in order to do so.\\
Recall that throughout, we assume that $A=\field Q/I$ is monomial, where the vertices of $Q$ are given by the set $\{1, \dots, n\}$ and that $\Sim(A)$ is equipped with the partial order $\leq_A$ corresponding to the natural order on the vertices of $Q$.
\begin{definition}
     We call an element $x\in e_j^AAe_i^A$ \textbf{direction-preserving} if $i< j$ and \textbf{non-direction-preserving} if $i\geq j$. Denote by $B_{\max}$ the subalgebra of $A$ generated by  $L$ and all direction-preserving paths.\\
    We call $x\in A$ \textbf{right-minimal direction-preserving} if $x\in B_{\max}$ \\
    and ${x\notin \rad(A)\rad(B_{\max})}$.
    Moreover, we let $B_{\min}$ be the subalgebra of $A$ generated by all right-minimal direction preserving paths.
\end{definition}
A direction-preserving path $p$ is a path going from a vertex $i=s(p)$ to a vertex $j=t(p)$ with $j>i$. We can visualize this as follows:
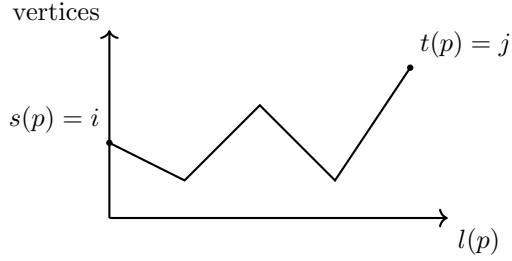
\begin{center}
\begin{tikzpicture}
    \draw[thick,->] (0,0) -- (4.5,0) node[anchor=north west] {$l(p)$};
\draw[thick,->] (0,0) -- (0,2.5) node[anchor=south east] {vertices};
\draw[thick] (0,1) node[anchor= south east] {$s(p)=i$} -- (1,0.5) -- (2,1.5) -- (3, 0.5) -- (4, 2) node[anchor = south west] {$t(p)=j$};
\filldraw (0,1) circle (1pt);
\filldraw (4,2) circle (1pt);
\end{tikzpicture}
\end{center}
On the other hand, a right-minimal direction-preserving path is a path going from a vertex $i=s(p)$ to a vertex $j=t(p)$ with $j>i$ and such that all inner vertices $k$ of $p$ are less than or equal to $i$. In other words, $j=t(p)$ is the only vertex appearing in $p$ which is bigger than $i$:
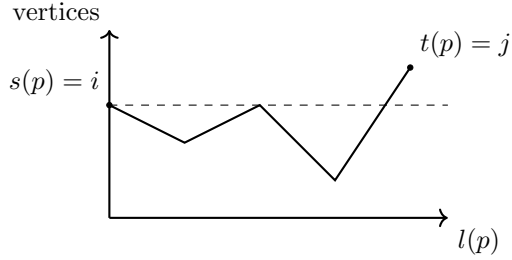
\begin{center}
\begin{tikzpicture}
    \draw[thick,->] (0,0) -- (4.5,0) node[anchor=north west] {$l(p)$};
\draw[thick,->] (0,0) -- (0,2.5) node[anchor=south east] {vertices};
\draw[thick] (0,1.5) node[anchor= south east] {$s(p)=i$} -- (1,1) -- (2,1.5) -- (3, 0.5) -- (4, 2) node[anchor = south west] {$t(p)=j$};
\filldraw (0,1.5) circle (1pt);
\filldraw (4,2) circle (1pt);
\draw[dashed] (0,1.5) -- (4.5, 1.5);
\end{tikzpicture}
\end{center}
We easily observe the following fact:
\begin{remark}\label{remark_borel_is_subset}
    Suppose $(A, \leq_A)$ is quasi-hereditary and $B$ is an exact Borel subalgebra of $(A, \leq_A)$ containing $L$. Then, since $B$ is directed, $B\subseteq B_{\max}$.
\end{remark}
On the other hand, it is also easy to note the following:
\begin{lemma}\label{lemma_borel_includes}
    Suppose $(A, \leq_A)$ is a basic monomial quasi-hereditary algebra and $B$ is an exact Borel subalgebra which has a basis consisting of paths. Then $B$ contains every right-minimal direction-preserving path. In particular, any exact Borel subalgebra which has a basis consisting of paths contains the subalgebra $B_{\min}$ of $A$.
\end{lemma}
\begin{proof}
    First, note that since $B$ is an exact Borel subalgebra generated by paths, $B$ must contain all trivial paths.
    Let $p\in e_j^A Ae_i^A$ be a right-minimal direction-preserving path. Then, since $i<j$, $\Delta_i$ has no composition factor isomorphic to $L_j$, so that $e_j^A \Delta_i=0$. Thus $p\Delta_i=(0)$. Since $\Delta_i\cong A\otimes_B L_i^B\cong Ae_i/A\rad(B)e_i$, we have $p\in A\rad(B)e_i$. In particular, there exists $b_1, \dots, b_n\in \rad(B)$ and $a_1, \dots a_n\in A$ such that \\
    $p=\sum_{l=1}^n a_l b_l$. Since both $ \rad(B)$ and $\rad(A)$ have a basis given by paths, we can assume $b_l$ and $a_l$ to be paths, by decomposing them further if necessary. Since $p$ is a path, we thus obtain $a_lb_l=0$ for all but one $l$. Hence $p=a_lb_l$, where $b_l$ is a path $b_l\in e_s^A\rad(B)e_i^A$ and $a_l$ is a path $a_l\in e_j^AAe_s^A$ for some $s$. Now since $p\notin \rad(A)\rad(B_{\max})$ and $\rad(B)\subseteq \rad(B_{\max})$, we have that $a_l\notin \rad(A)$. Thus $a_l=e_j=e_s$, so that $b_j=p\in B$.
\end{proof}
In the remainder of this section, we will show that $B_{\min}$ is in fact an exact Borel subalgebra of $(A, \leq_A)$ whenever $(A, \leq_A)$ is quasi-hereditary.
We begin by describing the paths in $B_{\min}$. Except for the trivial paths, these are always non-zero non-empty products of right minimal direction-preserving paths. As such, they graphically look somewhat like this:
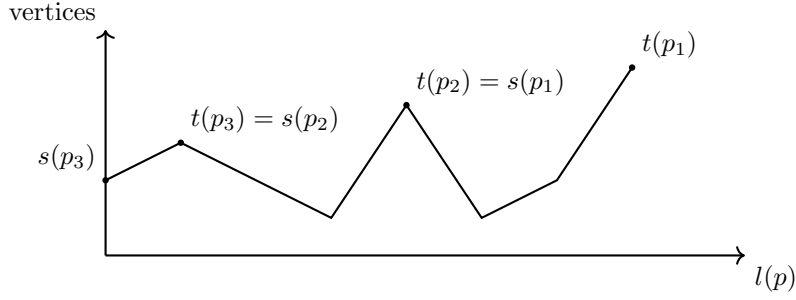
\begin{center}
\begin{tikzpicture}
    \draw[thick,->] (0,0) -- (8.5,0) node[anchor=north west] {$l(p)$};
\draw[thick,->] (0,0) -- (0,3) node[anchor=south east] {vertices};
\draw[thick] (0,1) node[anchor= south east] {$s(p_3)$} -- (1,1.5) node[anchor = south west] {$t(p_3)=s(p_2)$} -- (2,1) -- (3, 0.5) -- (4, 2) node[anchor = south west] {$t(p_2)=s(p_1)$} -- (5, 0.5) -- (6, 1)--  (7, 2.5) node[anchor = south west] {$t(p_1)$};
\filldraw (0, 1) circle (1pt);
\filldraw (1,1.5) circle (1pt);
\filldraw (4, 2) circle (1pt);
\filldraw (7, 2.5) circle (1pt);
\end{tikzpicture}
\end{center}
here for $p+I=(p_1+I)(p_2+I)(p_3+I)$. One can see that in such paths, the last vertex is always bigger than all other vertices that the path passes through. Let us prove this formally.
\begin{lemma}\label{lemma_product}
    Let $p+I=\prod_{i=1}^L (p_l+I)=(p_1+I)(p_2+I)\cdots (p_L+I)$ be a non-zero product of right-minimal direction preserving paths and $L\geq 1$. Then $k<t(p)$ for any interior vertex $k$ of $p$. In particular, since $p$ is directed, so that $s(p)<t(p)$, this implies that $\max(p)=t(p)$.
\end{lemma}
\begin{proof}
    We proceed by induction. If $p+I$ is a right-minimal direction-preserving path, then $s(p)<t(p)$, since $p+I$ is direction-preserving, and $s(p)>k$ for all interior vertices $k$ of $p$, since otherwise, $(p+I)=(p'+I)(p''+I)$, where $p'$ and $p''$ are non-trivial with $t(p'')=s(p')=k>s(p)$, so that $p''+I\in \rad(B_{\max})$ and $p'+I\in \rad(A)$.\\
    Now, suppose $p+I=\prod_{i=1}^L (p_l+I)$ for $L>1$ is a non-zero product of right-minimal direction-preserving paths. Let $p':=\prod_{i=1}^{L-1}p_l$. Then $p'+I=\prod_{i=1}^{L-1}(p_l+I)$ is a product of right-minimal direction-preserving paths and by induction hypothesis, $\max(p')=t(p')$ and $k<t(p')$ for any interior vertex $k$ of $p'$. On the other hand, $\max(p_L)=t(p_L)$, and $t(p_L)=s(p')<t(p')$, so that $k<t(p')$ for every $k\in V_{p_L}$. Since any interior vertex of $p$ is an interior vertex of $p'$ or a vertex of $p_L$, this implies that $k<t(p)$ for any interior vertex $k$ of $p$.
\end{proof}
In fact, the non-trivial paths in $B_{\min}$ are exactly the non-trivial paths in $A$ where the last vertex is bigger than all other ones that the path passes through. The reason is that, given a non-trivial path $p+I$ where the last vertex is bigger than all other vertices which $p$ passes through, we can cut $p+I$ into right-minimal direction preserving paths by inductively inserting a new cut-off point at the next vertex in $p$ which is bigger than the vertex at the previous cut-off point:
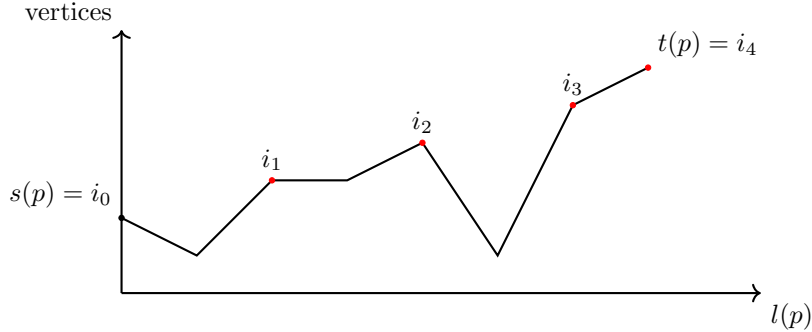
\begin{center}
\begin{tikzpicture}
    \draw[thick,->] (0,0) -- (8.5,0) node[anchor=north west] {$l(p)$};
\draw[thick,->] (0,0) -- (0,3.5) node[anchor=south east] {vertices};
\draw[thick] (0,1) node[anchor= south east] {$s(p)=i_0$} -- (1,0.5)  -- (2,1.5) node[anchor = south] {$i_1$} -- (3, 1.5) -- (4, 2) node[anchor = south] {$i_2$} -- (5, 0.5) -- (6, 2.5) node[anchor = south] {$i_3$} --  (7, 3) node[anchor = south west] {$t(p)=i_4$};
\filldraw (0,1) circle (1pt);
\filldraw[red] (2,1.5) circle (1pt);
\filldraw[red] (4, 2) circle (1pt);
\filldraw[red] (6,2.5) circle (1pt);
\filldraw[red] (7,3) circle (1pt);
\end{tikzpicture}
\end{center}
This gives rise to right-minimal direction preserving paths between any subsequent cut-off points, the product of which will be $p+I$. Note that last cut-off point will always be the terminal vertex, since it is bigger than any preceding vertex in $p$.\\
Let us give a more formal proof of this claim.
\begin{lemma}\label{lemma_bijection}
    Let $p+I$ be a non-trivial path in $A$ with $t(p)=i$. Then $p+I$ can be written as a product of right-minimal direction-preserving paths in $A$ if and only if $k< i$ for any interior vertex of $p$ and $p$ is directed.
\end{lemma}
\begin{proof}
    Suppose first that $p+I$ can be written as a product of right-minimal direction-preserving paths $p+I=\prod_{l=1}^k (p_l+I)=\prod_{l=1}^k p_l+I$ in $A$. Then $p$ is clearly direction-preserving, and by Lemma \ref{lemma_product}, $k< i$ for any interior vertex of $p$.\\
    On the other hand, suppose $p+I$ is a non-trivial direction-preserving path in $A$ with $k<i=t(p)$ for all interior vertices of $p$. Since $p+I$ is direction-preserving, it can be written as $(p+I)=(q_1+I)(p_1+I)$ where $p_1+I$ is a right-minimal direction-preserving subpath of $p+I$, and $q_1+I$ is a subpath of $p+I$. If $q_1$ is trivial, we are done. Otherwise, since $s(q_1)$ is an interior vertex of $p$, $s(q_1)< i=t(q_1)$, so that $q_1+I$ is direction-preserving. Thus it can be written as $q_1+I=(q_2+I)(p_2+I)$ where $p_2+I$ is a right-minimal direction-preserving subpath of $q_1+I$, and $q_2+I$ is a subpath of $q_1+I$. We can proceed by induction to see that we can write $q+I=\Pi_{i=1}^k (p_k+I)$ for some right-minimal direction-preserving paths $p_1+I, \dots, p_k+I$. 
\end{proof}
We are now ready to prove our main theorem:
\begin{theorem}\label{thm_borel_hereditary}
    Let $(A, \leq_A)$ be a basic quasi-hereditary monomial algebra. Then $B_{\min}$ is an exact Borel subalgebra of $A$ with a basis given by paths.
\end{theorem}
\begin{proof}
    Clearly, $B:=B_{\min}$ is a directed subalgebra of $A$ and has a basis given by paths.\\ Since $\End_B(L_i^A)\cong \field \cong \End_B(L_j^B)$ for all $1\leq i, j\leq n$, the proof of the first implication in \cite[Theorem A]{Koenig} applies word for word in our situation, despite the fact that we do not assume that $\field$ is algebraically closed. Thus, it suffices to show that the restrictions of the costandard modules of $A$ from left $A$-modules to left $B$-modules are isomorphic to the costandard modules of $B$.\\
    Let us choose the basis of $A=\field Q/I$ given by all paths in $A$, i.e. by all elements $p+I$, where $p$ is a path in $Q$ which is not in $I$, and the corresponding dual basis of $A^*$.
    Then the injective $A$-module $I_i=(e_i^A A)^{*}$ has a $\field$-basis given by  the restrictions to $e_i^AA$ of the dual of the trivial path $(e_i+I)^*$ and the duals $(p+I)^*$ of all non-trivial paths $p+I$ in $A$ such that $t(p)=i=\max(p)$, with left $A$-module structure given by
    \begin{align*}
         x\cdot (p+I)^*_{|e_i^A A}:e_i^A A\rightarrow \field, y\mapsto (p+I)^*(yx)
    \end{align*}
    for $x\in A$.
    Hence, the corresponding costandard module  $\nabla_i^A=\eta_{\leq i}I_i$  of $A$ has a $\field$-basis given by the the restrictions to $e_i^AA$ of the dual of the trivial path $(e_i+I)^*$ and the duals $(p+I)^*$ of all non-trivial paths $p+I$ in $A$ such that $t(p)=i=\max(p)$, and a left $A$-module structure given as above.
    In particular, there is a bijection between the simple composition factors of $\nabla_i^A$ with multiplicities and the paths $p+I$ with  $t(p)=i=\max(p)$ given by 
    \begin{align*}
        p+I\mapsto L_{s(p)}.
    \end{align*}
    Since $A$ is quasi-hereditary, the simple with index $i$ appears exactly once as a composition factor of $\nabla_i^A$. This implies that for all non-trivial paths $p+I$ in $A$ such that $t(p)=i=\max(p)$, we have that $s(p)<t(p)$. Additionally, for every interior vertex $k$ of $p$, we can consider a subpath $p'$ of $p$ starting at $k$ and ending at $t(p)$ to see that this also implies $k=s(p')<t(p')=t(p).$\\
    On the other hand, since $B$ is directed, the costandard module  $\nabla_i^B=I_i^B=(e_i^A B)^*$  of $B$  is just the injective hull of $L_i^B$, and thus it has a $\field$-basis given by the restriction to $e_i^AB$ of the dual of the trivial path $(e_i+I)^*$ and the duals $(p+I)^*$ of all non-trivial paths $p+I$ in $A$ with $t(p)=i$ which can be written as a product of right-minimal direction-preserving paths, with left $B$-module structure given by 
    \begin{align*}
         x\cdot (p+I)^*_{|e_i^A B}:e_i^A B\rightarrow \field, y\mapsto (p+I)^*(yx)
    \end{align*}
    for $x\in B$.
    By Lemma \ref{lemma_bijection}, there is a bijection between these two bases, which gives rise to  an isomorphism of vector spaces
    \begin{align*}
        \nabla_i^A\rightarrow \nabla_i^B, (p+I)^*_{|e_i^AA}\mapsto (p+I)^*_{|e_i^AB}.
    \end{align*}
    Clearly, this is an isomorphism of left $B$-modules, so that $\nabla_i^A\cong \nabla_i^B$ as left $B$-modules.
\end{proof}
We now obtain an exact Borel subalgebra $B_{\min}$ for any quasi-hereditary monomial algebra $(A, \leq_A)$. In a discussion with  Steffen König and Teresa Conde, we observed that passing from $A$ to $A^{\op}$, we also obtain a $\Delta$-subalgebra of $(A, \leq_A)$, and that one can therefore conclude the following:
\begin{theorem}\label{thm_reedy}
    Let $(A, \leq_A)$ be a basic quasi-hereditary monomial algebra. Let $C_{\min}$ be the exact Borel subalgebra of $(A^{\op}, \leq_A)$ generated by right-minimal direction preserving paths. Then $C_{\min}^{\op}$ is a $\Delta$-subalgebra of $A$. Moreover, there is a Reedy decomposition
    \begin{align*}
        C_{\min}^{\op}\otimes_L B_{\min}\rightarrow A, c\otimes b\mapsto cb.
    \end{align*}
\end{theorem}
\begin{proof}
    Since $C_{\min}$ is an exact Borel subalgebra of $(A^{\op}, \leq_A)$ by Theorem \ref{thm_borel_hereditary}, \cite[Theorem B]{Koenig} implies that  $C_{\min}^{\op}$ is a $\Delta$-subalgebra of $A$. Moreover, by Lemma \ref{lemma_bijection}, $C_{\min}^{\op}$ is spanned, as a vector space, by $L$ and all paths $p+I$ in $A$ such that $s(p)=\max(p)>k$ for $k=t(p)$ as well as for all $k\in \Inn(p)$, and $B_{\min}$ is spanned by $L$ and all paths $p+I$ in $A$ such that $t(p)=\max(p)>k$ for $k=s(p)$ as well as for all $k\in \Inn(p)$. Hence it is clear that $C_{\min}^{\op}\cap B_{\min}=L$, so that by \cite[Theorem 4.1]{Koenig2}, we obtain the Reedy decomposition from above.
    \end{proof}
    In particular, if $(A, \leq_A)$ is quasi-hereditary, then any path in $A$ can be written as a product of a path in $C$ and a path in $B$. Therefore, we obtain the following corollary, which gives an even easier description of the paths in $B_{\min}$:
\begin{corollary}\label{corollary_span}
    Suppose $(A, \leq_A)$ is quasi-hereditary, and let $p+I$ be a  path in $A$. Then $\max(p)$ appears exactly once as a vertex of $p$, i.e. the only subpath $q$ of $p$ with $s(q)=\max(p)=t(q)$ is the trivial path at $\max(p)$.\\
    In particular,  if $\max(p)=t(p)$ for some non-trivial path $p+I$ in $A$, then $k<t(p)$ for all $k\in \Inn(p)$ and $s(p)<t(p)$. Thus, $B_{\min}$ is spanned, as a vector space, by all paths $p+I$ in $A$ such that $t(p)=\max(p)$.
\end{corollary}
    \begin{remark}
        Let $(A=\field Q/I, \leq_A)$ be a monomial algebra equipped with a partial order on its set of vertices, and suppose that $Q$ contains no directed cycles. Then the above Corollary is statisfied a priori, and in this case, it can be easily seen that \cite[Theorem 3.10]{schroll} can be reformulated to state that $(A, \leq_A)$ is quasi-hereditary if and only if the map 
        \begin{align*}
        C_{\min}^{\op}\otimes_L B_{\min}\rightarrow A, c\otimes b\mapsto cb
    \end{align*}
    is an isomorphism. Since Reedy algebras are quasi-hereditary by \cite[Theorem 4.20]{DaleziosStovicek}, this gives an alternative proof of our main theorem in the case that $Q$ contains no directed cycles.
    \end{remark}
Using the $\Delta$-subalgebra $C_{\min}^{\op}$, we obtain an explicit way 
of writing $A$ as a direct sum of projective $B_{\min}$-modules. 
\begin{corollary}\label{proposition_decomposition}
    Let $(A=\field Q/I, \leq_A)$ be a basic quasi-hereditary monomial algebra and let $B_{\min}$ be the exact Borel subalgebra of $A$ generated by the right-minimal direction-preserving paths. Let $\mathcal{P}$ be the set of all paths $q+I$ in $A$ such that $s(q)=\max(q)$, including the trivial paths. Then $A\cong \bigoplus_{q+I\in \mathcal{P}}(q+I)B_{\min}$ as right $B_{\min}$-modules, where $(q+I)B_{\min}\cong e_{s(q)}^A B_{\min}$ for all $p+I\in \mathcal{P}$.
\end{corollary}
\begin{proof}
Let $C_{\min}$ be the exact Borel subalgebra of $(A^{\op}, \leq_A)$ generated by right-minimal direction-preserving paths. Then by Corollary \ref{corollary_span}, $C_{\min}^{\op}$ is spanned, as a vector space, by the paths in $\mathcal{P}$.  Since $L\cong \field^n$, this gives rise to a decomposition $C_{\min}^{\op}\cong \bigoplus_{p\in \mathcal{P}}\field p$ of $C_{\min}^{\op}$ as a projective right $L$-module, which in turn gives rise to a decomposition  $C_{\min}^{\op}\otimes_L B_{\min}\cong \bigoplus_{p\in \mathcal{P}}\field p\otimes_L B$ of $C_{\min}^{\op}\otimes_L B_{\min}$ as a projective right $B$-module. Using that multiplication in $A$ induces an isomorphism ${C_{\min}^{\op}\otimes_L B_{\min}\rightarrow A}$ of right $B_{\min}$-modules by \ref{thm_reedy}, we obtain the desired decomposition of $A$. 
\end{proof}
In \cite{Conde2}, it was shown that every exact Borel subalgebra is automatically normal. In the following, we use the above decomposition to exhibit the normality of $B_{\min}$ by giving a concrete splitting of the inclusion map $B_{\min}\rightarrow A$.
\begin{corollary}
    Let $(A, \leq_A)$ be a basic monomial quasi-hereditary algebra and let $B_{\min}$ be the exact Borel subalgebra of $A$ generated by the right-minimal direction-preserving paths. Let $\mathcal{P}$ be as in Corollary \ref{proposition_decomposition}. Then the map
    \begin{align*}
       \pi: A=\bigoplus_{q+I\in \mathcal{P}}(q+I)B_{\min}\rightarrow B_{\min}\\
        (q+I)b\mapsto 0\textup{ if }q+I\notin\{e_1^A, \dots, e_n^A\}\\
        (q+I)b\mapsto (q+I)b \textup{ if }q+I\in\{e_1^A, \dots, e_n^A\}.\\
    \end{align*}
    is a splitting of the canonical inclusion $B_{\min}\rightarrow A$ as right $B_{\min}$-modules and $\ker(\pi)$ is a right ideal in $A$
\end{corollary}
\begin{proof}
    Note that $\ker(\pi)=\bigoplus_{q+I\in \mathcal{P}\setminus\{e_1^A, \dots, e_n^A\}}(q+I)B$. Let $q+I\in\mathcal{P}\setminus\{e_1^A, \dots, e_n^A\}$ and $p+I\in A$ be a path in $A$ such that $(q+I)(p+I)\neq 0$. Then, since the product $(q+I)(p+I)=qp+I$ is a path in $A$ and $B_{\min}$ has a basis given by paths, it suffices to show that $qp+I\notin \bigoplus_{i=1}^n e_i^A B_{\min}=B_{\min}$. Recall that by Lemma \ref{lemma_product}, if $qp+I\in B_{\min}$, then $s(qp)<t(qp)$ and $k<t(qp)$ for any interior vertex of $qp$. However, by assumption $t(qp)=t(q)\leq s(q)=\max(q)$, a contradiction. Hence $qp+I\notin B_{\min}$, so that $\ker(\pi)$ is a right ideal in $A$.
\end{proof}
In general, the exact Borel subalgebra $B_{\min}$ is not regular, on the contrary, there exist path algebras with orders that don't admit any regular exact Borel subalgebras, see for example \cite[Theorem 60]{Markus}. In fact, if $\field$ is algebraically closed, then by \cite[Theorem 8.4]{uniqueness}, the exact Borel subalgebra $B_{\min}$ is regular if and only if $\field Q$ admits a regular exact Borel subalgebra, which can be checked using \cite[Theorem D]{Conde}. We will further investigate the question of regularity in Section \ref{section5}, under the assumption that $\field$ is algebraically closed.
\begin{example}\label{counterexample}
 The following example shows how one can use the above theorem to find an exact Borel subalgebra of $\field Q$. Moreover, it illustrates that this subalgebra is not regular in general, as well as that $\field Q$ may admit other exact Borel subalgebras containing $L=\bigoplus_{i\in V_Q}\field e_i$ which are not generated by paths. Additionally, it provides a counterexample for \cite[Theorem 2.1]{yuehui_zhang2} and \cite[Theorem 3.6]{yuehui_zhang}.
    Let $Q$ be the quiver 
    \begin{center}
\begin{tikzcd}[ampersand replacement=\&]
	1 \& 3 \& 2
	\arrow["\alpha", from=1-1, to=1-2]
	\arrow["\gamma", curve={height=-18pt}, from=1-1, to=1-3]
	\arrow["\beta", from=1-2, to=1-3]
\end{tikzcd} 
  \end{center}
    and $A=\field Q$. Then the direction-preserving paths in $Q$ are $\alpha, \gamma$ and $\beta\alpha$, of which $\alpha$ and $\gamma$ are right-minimal direction-preserving. Hence $B=\field Q'$ with $Q'$ given by the quiver
    \begin{center}
\begin{tikzcd}[ampersand replacement=\&]
	1 \& 3 \& 2
	\arrow["\alpha", from=1-1, to=1-2]
	\arrow["\gamma", curve={height=-18pt}, from=1-1, to=1-3]
\end{tikzcd}
 \end{center}
is an exact Borel subalgebra of $(A, \leq_A)$. And indeed, as right $B$-modules
\begin{align*}
    A=B\oplus \beta B\cong B\oplus e_3B,
\end{align*}
and
\begin{align*}
    A\otimes_B L_1^B\cong Ae_1/A\rad(B)e_1=Ae_1/\langle \alpha, \gamma, \beta\alpha\rangle\cong L_1^A\cong \Delta_1^A\\
    A\otimes_B L_2^B\cong Ae_2/A\rad(B)e_2=Ae_2=P_2^A\cong\Delta_2^A\\
     A\otimes_B L_3^B\cong Ae_3/A\rad(B)e_3=Ae_3=P_3^A\cong\Delta_3^A.
\end{align*}
Note that the subalgebra $B'$ which is spanned as a $\field$-vector space by the set\\ $\{e_1, e_2, e_3, \alpha, \gamma+\beta\alpha\}$ is also an exact Borel subalgebra such that $B\neq B'$ and $B'$ and $B$ contain a common maximal semisimple subalgebra $L:=\bigoplus_{i=1}^3 \field e_i$ of $A$. In fact, as we will see in the following, $B$ and $B'$ are not conjugated to each other, although there clearly is an automorphism of $A$ mapping one to the other. To see this, note that $\gamma \notin B'$ and for every invertible element $a\in A$ we can write \begin{align*}
    a=c_1e_1+c_2e_2+c_3e_3+c_{\alpha}\alpha+c_{\beta}\beta+c_{\gamma}\gamma+c_{\beta\alpha}\beta\alpha,
\end{align*}
and
 \begin{align*}
    a^{-1}=c_1'e_1+c_2'e_2+c_3'e_3+c_{\alpha}'\alpha+c_{\beta}'\beta+c_{\gamma}'\gamma+c_{\beta\alpha}'\beta\alpha,
\end{align*}
for some  $c_{*}, c'_{*}\in \field$ for $*\in \{1,2,3,\alpha, \beta, \gamma, \beta\alpha\}$ such that $c_i\neq 0\neq c_i'$ for $1\leq i\leq 3$, and obtain that
\begin{align*}
   a\gamma a^{-1}=c_1'c_3\gamma,
\end{align*}
so that $a\gamma a^{-1}\in \field \gamma\setminus \{0\}$ and hence $a\gamma a^{-1}\notin B$.\\
Note that $B$ is not regular: $\Ext^1_A(\Delta_1^A, \Delta_3^A)=\Ext^1_A(L_1^A, P_3^A)$ is two-dimensional, since it is spanned by the extension $E_1:=P_1/\langle \gamma\rangle$, whose Loewy diagram is 
\begin{center}
\begin{tikzcd}[ampersand replacement=\&]
	1 \\
	3 \\
	2
	\arrow["\alpha", from=1-1, to=2-1]
	\arrow["\beta", from=2-1, to=3-1]
\end{tikzcd}
\end{center}
and the extension $E_2:=(P_1\oplus P_3)/\langle (\beta\alpha, 0), (\alpha, 0), (\gamma, -\beta)\rangle$, whose Loewy diagram is
\begin{center}
\begin{tikzcd}[ampersand replacement=\&]
	1 \&\& 3 \\
	\& 2
	\arrow["\gamma", from=1-1, to=2-2]
	\arrow["\beta"', from=1-3, to=2-2]
\end{tikzcd}
\end{center}
However, $\Ext^1_B(L_1^B, L_3^B)$ is one-dimensional.
In particular, this does not contradict the uniqueness result in \cite[Theorem 8.4]{uniqueness}.
\end{example}
   
 \section{Uniqueness}\label{section4}
In the previous section, we showed that the subalgebra $B_{\min}$ of a quasi-hereditary monomial algebra $(A, \leq_A)$ generated by the right-minimal direction preserving paths is an exact Borel subalgebra.
In this subsection we show that any exact Borel subalgebra $B$ of $A$ which has a basis given by paths must be generated by the right-minimal direction-preserving paths. In particular, this shows that if $Q$ is a tree, then any exact Borel subalgebra of $A$ must be conjugated to $B_{\min}$.\\
One possible way to do this is to show that that the subalgebra $B_{\min}$ generated by right-minimal direction preserving paths is an exact Borel subalgebra of $B$ and to use this to show that $B=B_{\min}$. This can be done as in the second half of the proof of \cite[Theorem 8.6]{uniqueness}.\\
Here, we chose to give an alternative, more direct proof, where we first establish that if $A$ is a basic monomial quasi-hereditary algebra and $B$ is an exact Borel subalgebra of $A$, then $A$ decomposes as a right $B$-module as $A=\bigoplus_{p\in P}pB$ for some set $P$ of paths in $A$. Recall that in the case where $B=B_{\min}$, we have already seen such a decomposition in Corollary \ref{proposition_decomposition}. 
\begin{lemma}\label{lemma_rewrite}
    Let $B$ be a finite-dimensional algebra and let $e_1, \dots, e_n$ be a set of primitive orthogonal idempotents in $B$ such that $1_B=\sum_{i=1}^n e_i$. Let $X$ be a  projective right $B$-module and suppose we have a decomposition $X=P\oplus (x+y)B$ where $P$ is a projective right $B$-module and $x,y\in X$ such that $$(x+y)B=(x+y)e_iB\cong e_iB$$ for some $1\leq i\leq n$. Then we have a decomposition  $X=P\oplus xB$ or a decomposition $X=P\oplus yB$.
\end{lemma}
\begin{proof}
    Since $x, y\in X$ there are $p_x, p_y\in P$ and $b_x, b_y\in e_iB$ such that 
    \begin{align*}
        x=p_x+(x+y)b_x\\
        y=p_y+(x+y)b_y.
    \end{align*}
    Then we have
    \begin{align*}
        x+y=p_x+p_y+(x+y)(b_x+b_y).
    \end{align*}
    Since $X=P\oplus (x+y)B$, and $p_x+p_y\in P$, as well as $x+y, (x+y)(b_x+b_y)\in (x+y)B$, we obtain that $p_x+p_y=0$ and $(x+y)(b_x+b_y)=x+y$. Since $(x+y)B\cong e_iB$ the epimorphism $$e_iB\rightarrow (x+y)B, b\mapsto (x+y)b$$ is an isomorphism, so that  $(x+y)(b_x+b_y)=x+y$ implies $b_x+b_y=e_i$.
    \\
    Write $b_x=\lambda_x e_i+ r_x$ and $b_y=\lambda_y e_i +r_y$ for $\lambda_x, \lambda_y\in \field$ and $r_x, r_y\in e_i\rad(B)$. Then if $\lambda_x\neq 1$, then $1-\lambda_x e_i=(1-\lambda_x) e_i+ \sum_{j\neq i} e_j$ is invertible in $B$ with inverse $(1-\lambda_x)^{-1} e_i+ \sum_{j\neq i} e_j$, and as $r_x\in \rad(B)$, so is $1-\lambda_x e_i-r_x=1-b_x$. Therefore,
    \begin{align*}
        x(1-b_x)=p_x+yb_x\in P+ yB
    \end{align*}
    implies that $x\in P+yB$. Hence $X=P+yB$. Since the map $$e_iB\rightarrow yB=ye_iB, x\mapsto yx$$ is an epimorphism, the $\field$-dimension of $yB=ye_iB$ is at most that of $e_iB\cong (x+y)B$. This implies that $X=P\oplus yB$. On the other hand, if $\lambda_x=1$, then $\lambda_y=0$ since $b_x+b_y=e_i$. Thus,  $\lambda_y\neq 1$, so that similarly as before, $X=P\oplus xB$.
\end{proof}
\begin{lemma}\label{lemma_paths}
    Let $(A=\field Q/I, \leq_A)$ be a basic quasi-hereditary algebra and suppose $B$ is an exact Borel subalgebra of $A$ containing $L=\bigoplus_{i\in Q_0}\field e_i^A$. Then, $A$ decomposes as a right $B$-module
    \begin{align*}
        A=\bigoplus_{i=1}^N (p_i+I) B
    \end{align*}
    where $p_i$ is a path in $Q$, $N\in \mathbb{N}$ and $(p_i+I) B\cong e_{s(p_i)}^AB$ as right $B$-modules.
\end{lemma}
\begin{proof}
    By assumption, we have that $A$ is projective as a right $B$-module. Let $A=\bigoplus_{i=1}^N x_i B$, where $x_i\in A$, $N\in \mathbb{N}$ and $x_iB\cong e_i^A B$ as a right $B$-module.
    Let $x_1=\sum_{j=1}^m p_j+I$, where $p_1,\dots, p_m$ are paths in $Q$ which start in the vertex $i$.
    Set $P=\bigoplus_{i=2}^N x_i B$, $x=p_1+I$ and $y=\sum_{j=2}^m p_j+I$. Then we have a decomposition $A=P\oplus (x+y)B$. Thus, by the previous lemma, we also have a decomposition $A=P\oplus xB$ or $A=P\oplus yB$. Now we can proceed by induction.
\end{proof}
In the following, we fix such a decomposition of $A=\bigoplus_i^N (p_i+I)B$.\\
Note that for every $1\leq j\leq n$, $e_j\in A=\bigoplus_i^N (p_i+I)B$ so that there is $1\leq i\leq N$ such that $e_j\in (p_i+I)B$ where $p_i$ is a path in $Q$.
But this implies that $p_i+I=e_j^A$, so that $e_1^A,\dots, e_n^A\in \{p_i|1\leq i\leq N\}$. Rearranging if necessary, we can and will therefore assume from now on that $p_i=e_i$ for $1\leq i\leq n$, while $p_l$ is non-trivial for $l>n$.
Now let $X=\bigoplus_{i=n+1}^L (p_i+I)B$ so that $A=B\oplus X.$
\begin{proposition}\label{proposition_indirection-preserving}
    Let $p+I\in e_j^A Ae_i^A $ be a non-direction-preserving and non-trivial path. Then $p+I\in X$.
\end{proposition}
\begin{proof}
    Let $p+I\in e_j^AAe_i^A$ be non-direction-preserving and non-trivial. Recall that $A$ decomposes as a right $B$-module into a direct sum 
    \begin{align*}
        A=\bigoplus_{l=1}^N (p_l+I)B=B\oplus X
    \end{align*}
    where where $p_i+I=e_i^A$ for $1\leq i\leq n$ and $p_l+I\in e_{t_l}^AAe_{s_l}^A$ are non-trivial paths for $l>n$.
    In particular, we have $b_1, \dots, b_n\in B$ such that
    \begin{align*}
        p+I=\sum_{l=1}^N (p_l+I)b_l=\sum_{l=1}^N e_j^A (p_l+I)b_l e_i^A.
    \end{align*}
    Since $p_l+I\in \rad(A)$ for $l>n$ and $p+I\in \rad(A)$, we conclude that
    $$\sum_{l=1}^n (p_l+I)b_l=\sum_{l=1}^n e_l^A b_i\in \rad(B),$$
    so that we can assume without loss of generality that $b_l\in \rad(A)\cap B=\rad(B)$ for $l=1, \dots , n$.
    Since every element in $\rad(B)$ is a sum of direction-preserving elements, this implies that $e_j^A (p_l+I)b_l e_i^A=0$ for $l\leq n$.
    Thus, 
    \begin{equation*}
        p+I=\sum_{l=n+I}^N e_j^A (p_l+I)b_l e_i^A\in X.\qedhere
    \end{equation*}
\end{proof}
\begin{proposition}\label{proposition_uniqueness}
     Suppose $(A, \leq_A)$ is a basic monomial quasi-hereditary algebra and $B$ is an exact Borel subalgebra containing all trivial and all right-minimal direction-preserving paths. Then $B$ is generated by the right-minimal direction-preserving paths.
\end{proposition}
\begin{proof}
    Let $B_{\min}$ be the subalgebra of $A$ generated by the right-minimal direction-preserving paths.
    Then $B_{\min}\subseteq B$. Let $b\in B$. Then, by Corollary \ref{proposition_decomposition}, we can write $b=\sum_{q+I\in \mathcal{P}}(q+I)b_q$ where $\mathcal{P}$ is the set of all paths $q+I$ in $A$ such that $s(q)=\max(q)$ and $b_q\in B_{\min}$ for every $q+I\in \mathcal{P}$.
    By Proposition \ref{proposition_indirection-preserving}, we know that if $q+I$ is non-trivial, then $q+I\in X$, so that, since $b_q\in B_{\min}\subseteq B$ and $X$ is right $B$-module, $(q+I)b_q\in X$. Hence we have $$b=b'+x$$ for  $x=\sum_{q+I\in \mathcal{P}\setminus \{e_i^A|1\leq i\leq n\}}(q+I)b_q\in X$ and $b'=\sum_{i=1}^n e_i^A b_{e_i}\in B_{\min}$.
    Since $B_{\min}\subseteq B$ and $A=B\oplus X$, this implies that $x=0$ so that $b=b'\in B_{\min}$.
\end{proof}
\begin{corollary}
    Suppose $(A, \leq_A)$ is a basic monomial quasi-hereditary algebra and $B$ is an exact Borel subalgebra with a basis consisting of paths. Then $B$ is generated by the trivial and the right-minimal direction-preserving paths. 
\end{corollary}
\begin{proof}
    Since $B$ is an exact Borel subalgebra generated by paths, $B$ contains all trivial paths. Moreover, by Lemma \ref{lemma_borel_includes}, $B$ contains every right-minimal direction preserving path. Hence this is a direct consequence of Proposition \ref{proposition_uniqueness}.
\end{proof}
\begin{corollary}
    Let  $(A=\field Q/I, \leq_A)$  be a basic algebra such that $Q$ is a tree.
    Since $Q$ has no parallel paths, $A$ is monomial, so let $B_{\min}$ be the subalgebra of $A$ generated by the right-minimal direction-preserving paths.\\
    Suppose $B$ is any exact Borel subalgebra of $A$. 
    Then there is an $a\in A^{\times}$ such that $B=aB_{\min}a^{-1}$. 
\end{corollary}
\begin{proof}
Suppose $Q$ has $n$ vertices $\{1, \dots, n\}$ and denote by $e_i\in \field Q$ the empty path at the vertex $i$ and by $e_i^A:=e_i+I$ the corresponding idempotent in $A$.\\
    Since $B$ is an exact Borel subalgebra of $A$, it has as many simple modules as $A$. Moreover, since $B$ is a finite-dimensional algebra, there is a bijection between simple $B$-modules up to isomorphism and indecomposable projective $B$-modules up to isomorphism. 
    Because all indecomposable projective $B$-modules are direct summands of $B$, this implies that $B$ has at least $n$ indecomposable summands as a left $B$-module, so that $B\cong \End_B(B)^{\op}$ contains a copy $L_B$ of $\field^n$. Since $\field^n$ is separable and $A=\bigoplus_{i=1}^n \field e_i\oplus \rad(A)$, where $L:=\bigoplus_{i=1}^n \field e_i\cong \field^n$, the Wedderburn--Malcev theorem as stated in \cite[Theorem 1]{Farnsteiner} yields an  $a\in A^{\times}$ such that such that $L_B=a^{-1}La$, so that
    $L=\bigoplus_{i=1}^n\field e_i\subseteq aBa^{-1}=:B'$ \\
    Now, since $Q$ is a tree and and $L\subseteq B'$, $B'$ has a basis consisting of paths. Moreover, since $B$ is an exact Borel subalgebra of $A$, so is $B'$. Hence $B'=B_{\min}$ by Proposition \ref{proposition_uniqueness}. 
\end{proof} 

 \section{Regularity}\label{section5}
 Regular exact Borel subalgebras play a distinguished role in the study of exact Borel subalgebras, primarily because of the existence result in \cite{KKO} and the uniqueness result in \cite{Miemietz}. They have been further studied for example in \cite{Conde, BKK}.
 Here, we investigate when the exact Borel subalgebra $B_{\min}$ of a fixed quasi-hereditary monomial algebra $A$ is a regular exact Borel subalgebra. 
If we assume $\field$ to be algebraically closed, then by \cite[Theorem 8.4]{uniqueness}, this is equivalent to $A$ admitting a regular exact Borel subalgebra. In \cite[Theorem D]{Conde}, several criteria for this were given. Here, we will use \cite[Point 7; Theorem D]{Conde}, which describes this in terms of a bijection between $\Ext^1_A(\Delta_i, \Delta_j)$ and $\Ext^1_A(\Delta_i, L_j)$. Our goal is to give a criterion which can be directly checked using paths in $Q$ and $I$; to this end, will explicitly describe  $\Ext^1_A(\Delta_i, \Delta_j)$ and $\Ext^1_A(\Delta_i, L_j)$ in these term.\\
Since all of the results we use ask for $\field$ to be algebraically closed, we will assume that this is the case from now on. Recall that as before, we assume that $A=\field Q/I,$ where $Q_0=\{1, \dots, n\}$, and that the partial order $\leq_A$ is induced by the natural order on the vertices of $Q$. Moreover, let us assume from now on that $(A, \leq_A)$ is quasi-hereditary.\\
\begin{definition}
    Let $1\leq i<j\leq n$. Then we define $E_{ij}$ as the set of non-zero right-minimal direction-preserving paths $p+I$ in $A$ such that $s(p)=i<t(p)=j$, and let $E_i:=\bigcup_{j=1}^n E_{ij}$ be the set of non-zero right-minimal direction-preserving paths $p+I$ in $A$ starting in $i$. Moreover, for $p\in E_{ij}$ let $E_{p}$ be the set of paths $q+I$ in $A$ with $s(q)=j$ and $qp+I=0$.
\end{definition}
\begin{example}\label{example_running}
    Let $Q$ be the quiver 
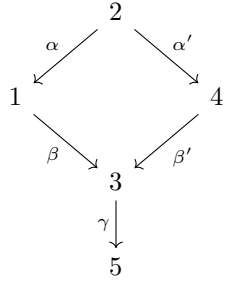
\[\begin{tikzcd}[ampersand replacement=\&]
	\& 2 \\
	1 \&\& 4 \\
	\& 3 \\
	\& 5
	\arrow["\alpha"', from=1-2, to=2-1]
	\arrow["{\alpha'}", from=1-2, to=2-3]
	\arrow["\beta"', from=2-1, to=3-2]
	\arrow["{\beta'}", from=2-3, to=3-2]
	\arrow["\gamma"', from=3-2, to=4-2]
\end{tikzcd}\]
let $I$ be the ideal in $\field Q$ generated by $\gamma\beta\alpha$, and let $A=\field Q/I$. Then $\beta\alpha+I\in E_{23}$, but $\beta'\alpha'\notin E_{23}$ since $\alpha'$ is already direction preserving. Moreover, $E_{\beta\alpha}=\{\gamma\}$ and $E_{\alpha'}=(0)$
\end{example}
\begin{lemma}
    Let $1\leq i\leq n$. Then, there is a short exact sequence
\[\begin{tikzcd}[ampersand replacement=\&]
	{\bigoplus_{p+I\in E_{i}}\bigoplus_{q+I\in E_p}P_{t(q)}} \& {\bigoplus_{p+I\in E_i}P_{t(p)}} \& {P_i} \& {\Delta_i} \& {(0)}
	\arrow["{\bigoplus(\sum r_q)}", from=1-1, to=1-2]
	\arrow["{\sum r_p}", from=1-2, to=1-3]
	\arrow["{\pi_i}", from=1-3, to=1-4]
	\arrow[from=1-4, to=1-5]
\end{tikzcd}\]
where $\pi_i$ is the canonical projection $\pi_i:P_i\rightarrow \Delta_i$ and $r_p$ denotes right multiplication by $p+I$.
\end{lemma}
\begin{proof}
Let $1\leq i, j\leq n$. Then, there is a vector space isomorphism
\begin{align*}
    e_j^AAe_i^A\rightarrow \Hom_A(P_j, P_i), a\mapsto r_a
\end{align*}
where $r_a:P_j=Ae_j^A\rightarrow P_i=Ae_i^A, x\mapsto xa$.
    Since $A=\field Q/I$ is a monomial algebra, there is a basis of $e_j^A A e_i^A$ given by all non-zero paths $p+I$ with $s(p)=i$ and $t(p)=j$. Moreover, if a path $p+I$ with $s(p)=i<t(p)$ is not right-minimal direction-preserving, then there are paths $p'+I$ and $p''+I$ in $A$ such that $p=p'p''$ and $p''+I$ is right-minimal direction preserving. In particular, $r_{p+I}=r_{p''+I}\circ r_{p'+I}$ so that $\im(r_{p+I})\subseteq \im(r_{p''+I})$.\\
    Let $\eta_i P_i:=\sum_{\varphi: P_j\rightarrow P_i, j>i}\im(\varphi)$, so that $\Delta_i^A=P_i/\eta_{>i}P_i$.
    Then, by the above, $\eta_{>i}P_i=\sum_{p\in E_i}\im(r_p)$, so that 
\[\begin{tikzcd}[ampersand replacement=\&]
	{\bigoplus_{p+I\in E_i}P_{t(p)}} \& {P_i} \& {\Delta_i} \& {(0)}
	\arrow["{\sum r_p}", from=1-1, to=1-2]
	\arrow["{\pi_i}", from=1-2, to=1-3]
	\arrow[from=1-3, to=1-4]
\end{tikzcd}\]
is a short exact sequence. Moreover, since $A$ is monomial, there is an exact sequence
\[\begin{tikzcd}[ampersand replacement=\&]
	{\bigoplus_{q+I\in E_p}P_{t(q)}} \& {\ker(r_p)} \& {(0)}
	\arrow["{\sum r_q}", from=1-1, to=1-2]
	\arrow[from=1-2, to=1-3]
\end{tikzcd}\]
for every $p\in E_i$.
Taking the direct sum of these exact sequences and wedging the resulting exact sequence with 
\[\begin{tikzcd}[ampersand replacement=\&]
	{\bigoplus_{p+I\in E_i}P_{t(p)}} \& {P_i} \& {\Delta_i} \& {(0),}
	\arrow["{\sum r_p}", from=1-1, to=1-2]
	\arrow["{\pi_i}", from=1-2, to=1-3]
	\arrow[from=1-3, to=1-4]
\end{tikzcd}\]
we obtain the desired exact sequence.
\end{proof}
\begin{example}
    Let $A$ be as in Example \ref{example_running}. Then the projective resolution for $\Delta_2^A=\begin{pmatrix}
        2 \\ 1
    \end{pmatrix}$ which we obtain here is
\[\begin{tikzcd}[ampersand replacement=\&]
	{P_5} \& {P_3^A\oplus P_4^A} \& {P_2^A} \& {\Delta_2^A}
	\arrow["{(r_{\gamma}, 0)}", from=1-1, to=1-2]
	\arrow["{r_{\beta\alpha}+r_{\alpha'}}", from=1-2, to=1-3]
	\arrow[from=1-3, to=1-4]
\end{tikzcd}\]
\end{example}
To describe extension spaces, let us recall the following, well-known fact:
\begin{lemma}\label{lemma_extensions}
   Let $M$ and $N$ be $A$-modules and let $P^\cdot(M)=(P^\cdot(M), d_M)$ and $P^{\cdot}(N)=(P^{\cdot}(N), d_N)$ be projective resolutions of $M$ and $N$ respectively. Denote by $\pi_M:P^0(M)\rightarrow M$ and $\pi_N:P^0(N)\rightarrow N$ the projections onto $M$ and $N$ respectively. Then, the following holds:
   \begin{enumerate}
       \item There is an epimorphism of vector spaces from the space of chain maps $(f_k)_k:P^\cdot(M)\rightarrow P^\cdot(N)[1]$ to the space of homomorphisms\\ $f: P^1(M)\rightarrow P^0(N)$ with $\pi_N\circ f\circ d_M^{(2)}=0$, which is given by $(f_k)_k\mapsto f_1$, the kernel of which consists of nullhomotopic maps.
       \item A chain map $(f_k)_k:P^\cdot(M)\rightarrow P^\cdot(N)[1]$ is nullhomotopic if and only if $$f_1=d_N^{(1)}\circ g_1-g_0\circ d_M^{(1)}$$ for some $A$-linear maps $g_1:P_M^1\rightarrow P_N^1$ and $g_0:P_M^0\rightarrow P_N^0$.
   \end{enumerate}
\end{lemma}
Using this, we can give an explicit description of $\Ext^1_A(\Delta_i^A, L_j^A)$:
\begin{lemma}
    There is an isomorphism of vector spaces $f:\field E_{ij}\rightarrow \Ext^1_A(\Delta_i^A, L_j^A)$ given by $p+I\mapsto [\pi_p]$, where $\pi_p$ is the canonical projection
    \begin{align*}
        \pi_p:\bigoplus_{q+I\in E_i}P_{t(q)}\rightarrow P_{t(p)}=P_j.
    \end{align*}
onto the direct summand corresponding to $p+I\in E_i$.
\end{lemma}
\begin{proof}
    Denote by $p_j:P_j\rightarrow L_j$ the canonical projection. In order to see that $\pi_p$ gives rise to a cocycle in $\Hom_A^{\cdot}(P^\cdot(\Delta_i^A), P^\cdot(L_j^A))$, we need to show that $$\pi_p(\ker(\sum_{p'+I\in E_i} r_{p'+I}))\subseteq \ker(p_j).$$ But this is true, since $p'+I\neq 0$ for $p'+I\in E_i$, so that $\ker(\sum_{p'+I\in E_i} r_{p'+I})$ is contained in the radical of $\bigoplus_{p'+I\in E_i}P_{t(p')}$, and thus 
    $$\pi_p(\ker(\sum_{p'+I\in E_i} r_{p'+I}))\subseteq \rad(P_j)=\ker(p_j).$$
    Moreover, since $\ker(p_j)=\rad(P_j)$, every morphism $f:\bigoplus_{q+I\in E_i}P_{t(q)}\rightarrow P_{t(p)}$ with $\im(f)\subseteq \rad(P_j)$ gives rise to a coboundary. Since $r_{p+I}$ is radical for every $p+I\in E_i$, we can conclude by \ref{lemma_extensions} that a morphism $f:\bigoplus_{q+I\in E_i}P_{t(q)}\rightarrow P_{t(p)}$ gives rise to a coboundary in $\Hom_A^{\cdot}(P^\cdot(\Delta_i^A), P^\cdot(L_j^A))$ if and only if $\im(f)\subseteq \rad(P_j)$. Hence the canoncial projections 
    \begin{align*}
        \pi_p:\bigoplus_{q+I\in E_i}P_{t(q)}\rightarrow P_{t(p)}=P_j.
    \end{align*}
    for all $p\in E_i$ give rise to a basis of $\Ext^1_A(\Delta_i^A, L_j^A)$, so that $f$ is an isomorphism of vector spaces.
\end{proof}
\begin{example}
    Let $A$ be as in Example \ref{example_running}. Then this lemma tells us that $\Delta_2^A=\begin{pmatrix}
        2 \\ 1
    \end{pmatrix}$ has two extensions by simple modules, corresponding to the right-minimal paths $\beta\alpha$ and $\alpha'$. The first one is an extension with $L_3^A$ and it corresponds to the module with Loewy diagram 
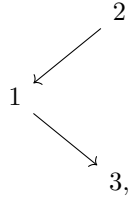
\[\begin{tikzcd}[ampersand replacement=\&]
	\& 2 \\
	1 \\
	\& 3,
	\arrow[from=1-2, to=2-1]
	\arrow[from=2-1, to=3-2]
\end{tikzcd}\] and the second is an extension by $L_4^A$ and it corresponds to the module with Loewy diagram 
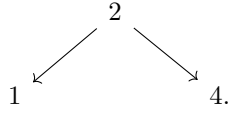
\[\begin{tikzcd}[ampersand replacement=\&]
	\& 2 \\
	1 \&\& 4.
	\arrow[from=1-2, to=2-1]
	\arrow[from=1-2, to=2-3]
\end{tikzcd}\]
\end{example}
In order to describe $\Ext^1_A(\Delta_i^A, \Delta_j^A)$, we need another definition:
\begin{definition}
    Let $E'_{ij}$ be the set of pairs of non-zero (possibly trivial) paths $(p+I, q+I)$ such that there is some $k\leq j$ fulfilling the following:
    \begin{enumerate}
        \item $p+I\in E_{ik}$,
        \item $s(q)=j=\max(q)\geq t(q)=k$.
        \item For any path $q'+I$ in $A$ with $s(q')=k$ such that $\max(q')\leq j$ we have $q'p\in I\Rightarrow q'q\in I$.
    \end{enumerate}
\end{definition}

\begin{lemma}
    There is an epimorphism of vector spaces $f':\field E_{ij}'\rightarrow \Ext^1_A(\Delta_i^A, \Delta_j^A)$
    given by $(p+I, q+I)\mapsto [r_q\circ \pi_p]$, where 
    \begin{align*}
        \pi_p:\bigoplus_{p'+I\in E_i}P_{t(p')}\rightarrow P_{t(p)}=P_k
    \end{align*}
    is the canonical projection onto the direct summand corresponding to $p+I\in E_i$, and 
     \begin{align*}
        r_q:P_k\rightarrow P_j, x\mapsto x(q+I)
    \end{align*}
    is given by right-multiplication given by $q+I$.
    Moreover, $$\ker(f')=\langle \sum_{p+I\in E_i; (p+I, pq''+I)\in E_{ij}'}(p+I, pq''+I)| q'':j\rightarrow i\textup{ path}\rangle.$$
\end{lemma}
\begin{proof}
    First, note that a basis of $\Hom_A(\bigoplus_{p\in E_i}P_{t(p)}, P_j)$ is given by morphisms of the form $r_q\circ \pi_p$, where $q+I$ is any path from $j$ to $t(p)$.\\
    Secondly, a morphism $r_q\circ \pi_p$ gives rise to a cocycle in $\Hom_A^{\cdot}(P^\cdot(\Delta_i^A), P^\cdot(\Delta_j^A))$ if and only if $$r_{q+I}\circ \pi_p(\ker(\sum_{p'+I\in E_i} r_{p'+I}))=\im (r_{q+I}\circ \pi_p\circ \sum_{q'+I\in E_p}r_{q'+I})\subseteq \ker(\pi_j)=\eta_{>j}P_j,$$ where $\eta_j P_j:=\sum_{\varphi: P_k\rightarrow P_j, k>j}\im(\varphi)$, which, since $A$ is monomial, is the case if and only if $\im(r_{q+I}\circ \pi_p\circ r_{q'+I})\subseteq \eta_{>j}P_j$ for all $q'+I\in E_p$, i.e. for all paths $q'+I\neq 0$ with $s(q')=t(p)$ and $q'p\in I$. If $\max(q')> j$, then the above always holds. On the other hand, if $\max(q')\leq j$, then, since $A$ is monomial, the above holds if and only if $q'q\in I$. Overall, we can conclude that  $r_q\circ \pi_p$ gives rise to a cocycle if and only if $q'p\in I\Rightarrow q'q\in I$ for all paths $q'+I$ in $A$ with $s(q')=t(p)$ and $\max(q')\leq j$.\\
    Finally, by Lemma \ref{lemma_extensions}, a linear generating set of the coboundaries is given by all maps $r_{q'+I}\circ r_{q''+I}\circ \pi_p$, where $p\in E_i$, $q'\in E_j$ and $q''+I$ is a path in $A$ such that $q''q'\notin I$ and $t(q'')=t(p)$; as well as by all maps $r_{q''+I}\circ \sum_{p+I\in E_i}r_{p+I}$, where $q''$ is any path from $j$ to $i$.\\
    Note that for any path $p+I\in E_i$ and $q:j\rightarrow i$, the pair $(p+I, pq''+I)$ is in $E_{ij}$ if and only if $\max(pq'')=j$ and $pq''\notin I$. Thus,  \begin{align*}
        r_{q''+I}\circ \sum_{p+I\in E_i}r_{p+I}=\sum_{p+I\in E_i; (p+I, pq''+I)\in E_{ij}'}r_{pq''+I}+ \sum_{p+I\in E_i, \max(pq'')>j}r_{pq''}.
    \end{align*}
    Hence $f'$ is surjective with kernel given by the span of the elements  $$\sum_{p+I\in E_i; (p+I, pq''+I)\in E_{ij}'}(p+I, pq''+I)$$ for $q'':j\rightarrow i$ with $j>i$.
\end{proof}
\begin{example}
     Let $Q$ be the quiver
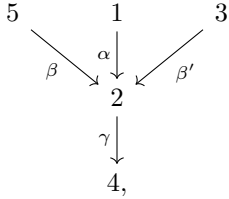
\[\begin{tikzcd}[ampersand replacement=\&]
	5 \& 1 \& 3 \\
	\& 2 \\
	\& 4,
	\arrow["\beta"', from=1-1, to=2-2]
	\arrow["\alpha"', from=1-2, to=2-2]
	\arrow["{\beta'}", from=1-3, to=2-2]
	\arrow["\gamma"', from=2-2, to=3-2]
\end{tikzcd}\]
let $I$ be the ideal generated by $\alpha\gamma$ and let $A=\field Q/I$.
Then the pair $(\alpha+I, \beta'+I)$ gives rise to an extension between $\Delta_1^A=L_1^A$ and $\Delta_3^A=\begin{pmatrix}
    3 \\ 2
\end{pmatrix}$ given by 
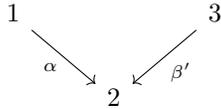
\[\begin{tikzcd}[ampersand replacement=\&]
	1 \&\& 3 \\
	\& 2
	\arrow["\alpha"', from=1-1, to=2-2]
	\arrow["{\beta'}", from=1-3, to=2-2]
\end{tikzcd}\]
and, accordingly, we have $(\alpha+I, \beta'+I)\in E_{1,3}'$.
On the other hand, the fact that  $\max(\gamma)=3\leq 5=s(\beta)$ and $\gamma\beta\notin I$, while  $\gamma\alpha\in I$, prevents a module with Loewy diagram
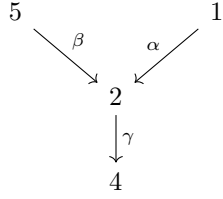
\[\begin{tikzcd}[ampersand replacement=\&]
	5 \&\& 1 \\
	\& 2 \\
	\& 4
	\arrow["\beta", from=1-1, to=2-2]
	\arrow["\alpha"', from=1-3, to=2-2]
	\arrow["\gamma", from=2-2, to=3-2]
\end{tikzcd}\]
from being an extension between $\Delta_1^A$ and $\Delta_5^A$, so that accordingly, $(\alpha+I, \beta+I)\notin E'_{1,5}$.
\end{example}
We are now ready to give a criterion for the existence of a regular exact Borel subalgebra of $A$ in terms of pairs of paths:
\begin{lemma}\label{lemma_regularity}
    $A$ has a regular exact Borel subalgebra if and only if
    \begin{align*}
        E_{ij}'=\{& (p+I, e_j^A)| p+I\in E_{ij}\}\\
        \cup\{&(p+I, pq''+I)| p\in E_i, q'':j\rightarrow i, s(q'')=\max(pq'') \textup{ s.t. }pq''\notin I\\&\textup{ and }p'q''\in I\textup{ for all }p'+I\in E_i\setminus\{p+I\} \textup{ with }\max(p'q'')=j\}
    \end{align*}
    for all $1\leq i< j\leq n$.
\end{lemma}
\begin{proof}
    By \cite[Theorem D]{Conde}, $A$ admits a regular exact Borel subalgebra if and only if the map $\Ext^1_A(\Delta_i^A,\pi_j):\Ext^1_A(\Delta_i^A, \Delta_j^A)\rightarrow \Ext^1_A(\Delta_i^A, L_j^A)$ is an isomorphism for all $1\leq i\leq j\leq n$.
    Now note that $\Ext^1_A(\Delta_i^A,\pi_j)(f'((p+I, q+I))=0$ if $q+I\neq e_j^A$ and  $\Ext^1_A(\Delta_i^A,\pi_j)(f'((p+I, e_j^A))=f(p+I)$ else.\\
    Thus, $\Ext^1_A(\Delta_i^A,\pi_j)$ is surjective if and only if $\{ (p+I, e_j^A)| p+I\in E_{ij}\}\subseteq E_{ij}'$ and injective if and only if $f'((p+I, q+I))=0$ for all $(p+I, q+I)\in E_{ij}'$ with $q+I\neq e_j^A$.\\
    Thus, it remains to be shown that the intersection $E_{ij}'\cap \ker(f')$ is the set of all pairs $(p+I, pq''+I)$, where $q'':j\rightarrow i$ is a path with $s(q'')=\max(pq'')$ and $p+I$ is a path in $E_i$ such that $p'q''\in I$ for all $p'+I\in E_i\setminus\{p+I\}$ with $\max(p'q'')=j$.\\
    Suppose $(p+I, q+I)\in \ker(f')$ for some $(p+I, q+I)\in E_{ij}'$. Then, $\max(q)=j=s(q)$ and there are pairwise distinct paths $q''_1, \dots, q''_L:j\rightarrow i$ and elements $\lambda_{1}\dots, \lambda_l\in \field\setminus\{0\}$ such that 
    \begin{align*}
        (p+I, q+I)=\sum_{l=1}^L \sum_{p'+I\in E_i; (p'+I, p'q''_l+I)\in E_{ij}'}\lambda_{l}(p'+I, p'q''_l+I).
    \end{align*}
    Since $A$ is monomial and $p'+I\neq 0\neq q'+I$ for all $(p'+I, q'+I)\in E_{ij}'$, we have that $(p'+I, p'q''_l+I)=(p''+I, p''q''_k+I)\in E_{ij}'$ if and only if $p'=p''$ and $q''_l=q''_k$.
    Hence, the sum consists of only one summand; in particular, $L=1$, and for $q'':=q_1'':j\rightarrow i$ we have
     \begin{align*}
        (p+I, q+I)=\sum_{p'+I\in E_i; (p'+I, p'q''+I)\in E_{ij}'}\lambda_{1}(p'+I, p'q''+I)=\lambda_1 (p+I, pq''+I)
    \end{align*}
    In particular, $\lambda_1=1$,
     $q=pq''$ and $(p'+I, p'q''+I)\notin E_{ij}'$ for all $p'+I\in E_i\setminus\{p+I\}$. Since $(p+I, pq''+I)\in E_{ij}'$, $s(q'')=\max(pq'')>t(q'')=s(p)$ and $pq''\notin I$. Moreover, we can see by definition of $E_{ij}'$ that $(p'+I, p'q''+I)\notin E_{ij}'$ for $p'+I\in E_i$ if and only if $p'q''\in I$ or $\max(p'q'')>j$. Thus, $(p+I, q+I)=(p+I, pq''+I)$ where $q'':j\rightarrow i$ with $s(q'')=\max(pq'')>t(q'')$, $pq''\notin I$ and $p'q''\in I$ for all $p'+I\in E_i\setminus\{p+I\}$ with $\max(p'q'')=j$.\\
    On the other hand, suppose that $q'':j\rightarrow i$ is a path such that\\ $s(q'')=\max(pq'')>t(q'')=s(p)$, $pq''\notin I$ and $p'q''\in I$  for all $p'+I\in E_i\setminus\{p+I\}$ with $\max(p'q'')=j$. Then $(p+I, pq''+I)\in E_{ij}'$, and \begin{align*}
        (p+I, pq''+I)=\sum_{p'+I\in E_i; (p'+I, p'q''+I)\in E_{ij}'}(p'+I, p'q''+I),
    \end{align*}
    so that $(p+I, pq''+I)\in\ker(f')$.
\end{proof}
In fact, one of the inclusions in the lemma above always holds trivially:
\begin{lemma}
    Let $1\leq i<j\leq n$ and $p+I\in E_{ij}$. Then $(p+I, e_j^A)\in E_{ij}'$.
\end{lemma}
\begin{proof}
    Consider the map $r_{p+I}: P_j\rightarrow P_i, x\mapsto x(p+I)$. This gives rise to a map
    \begin{align*}
        \overline{r_{p+I}}:&\Delta_j\rightarrow A\sum_{k=j}^ne_k^AAe_i^A/A\sum_{k=j+1}^ne_k^AAe_i^A,\\
        &x+A\sum_{k=j+1}^ne_k^AAe_j^A\mapsto x(p+I)+A\sum_{k=j+1}^ne_k^AAe_i^A.
    \end{align*}
    Since $A$ is quasi-hereditary, $A\sum_{k=j}^ne_k^AAe_i^A/A\sum_{k=j+1}^ne_k^AAe_i^A$ is isomorphic to a direct sum of copies of $\Delta_j$ by \cite[Lemma 1.4]{DlabRingel}.\\ Because $\End_A(\Delta_j)\cong \field$, this implies that $\overline{r_{p+I}}$ is either injective or zero. Since $p+I\neq 0$, it is not zero, so that it is injective.\\
    Recall that a basis of $\Delta_j$ is given by the elements $q'+I+A\sum_{k=j+1}^ne_k^AAe_j^A$, where $q'+I$ is a non-zero path in $A$ with $\max(q')=s(q')=j$. Thus, $q'p\notin I$ for all such $q'$, so that $(p+I, e_j^A)\in E_{ij}'$.
\end{proof}
Since we clearly have $(p+I, pq''+I)\in E_{ij}'$ for every path $q''+I$ in $A$ with $j=s(q'')=\max(pq'')>t(q'')=s(p)=i$ and $pq''\notin I$, it follows that, in general for $1\leq i<j\leq n$,
\begin{align*}
    E_{ij}'\supseteq \{& (p+I, e_j^A)| p+I\in E_{ij}\}\\
        \cup\{&(p+I, pq''+I)| p\in E_i, q'':j\rightarrow i, s(q'')=\max(pq'') \textup{ s.t. }pq''\notin I\\&\textup{ and }p'q''\in I\textup{ for all }p'+I\in E_i\setminus\{p+I\} \textup{ with }\max(p'q'')=j\}.
\end{align*}
Thus, we obtain the following corollary:
\begin{corollary}\label{corollary_criterion}
    $A$ has a regular exact Borel subalgebra if and only if for every pair $(p+I, q+I)\in E_{ij}'$ either $q+I=e_j^A$ or $q+I=pq''+I$ for some $q'':j\rightarrow i$ such that $p'q''\in I$ for all $p'+I\in E_i\setminus \{p+I\}$ with $\max(p'q'')=j$.
\end{corollary}
The criterion above is unfortunately still a bit complicated; in the following we give a simplification in the case that $A$ is a hereditary algebra.
\begin{theorem}
\label{proposition_regularity_hereditary}
    Suppose $I=(0)$. Then $A$ admits a regular exact Borel subalgebra if and only if the following two conditions hold:
    \begin{enumerate}
        \item \label{condition1} For all paths $p:i\rightarrow k$ and $q:j\rightarrow k$ with $\max(p)=k>i$ and $j>k$ there is a path $r:j\rightarrow i$ such that $q=pr$.
        \item \label{condition2} For any path $q:j\rightarrow i$ with $j>i$ there is at most one right-minimal direction preserving path $p$ starting in $i$ such that $t(p)< j$.
    \end{enumerate}
\end{theorem}
\begin{proof}
    Suppose $A$ admits a regular exact Borel subalgebra. Let us first show \eqref{condition1}.
    Suppose $p$ and $q$ are as in \eqref{condition1} above.
    Making $q$ shorter if necessary, we can assume without loss of generality that $j=\max(q)>k$. We proceed by induction on the length $l(p)$ of $p$. If $l(p)=1$ then $p\in E_i$, so that $(p, q)\in E_{ij}'$ and it follows from Lemma \ref{lemma_regularity} that $p$ is a subpath of $q$.\\
    Thus let us assume $l(p)>1$. If $p$ is right-mimimal direction-preserving, then, as in the length one case,  $(p, q)\in E_{ij}'$ and it follows from Lemma \ref{lemma_regularity} that $p$ is a subpath of $q$.\\
    Therefore, let us assume that $p$ is not right-minimal direction-preserving. Let $k'$ be the maximal inner vertex of $p$, so that $k'>i$. Let us write $p=p'p''$ where $p'$ and $p''$ are non-trivial paths with $s(p')=t(p'')=k'$.  Then $p'$ has no inner vertex greater than $k'$. Moreover, since $p'$ is non-trivial and $Q$ has no loops, $k'=s(p')\neq t(p')=k=\max(p')$. Therefore, $p'\in E_{k'}$ and $(p', q)\in E_{k'j}'$, so that it follows from Lemma \ref{lemma_regularity} that $p'$ is a subpath of $q$, that is, there is $r':j\rightarrow k'$ such that  $q=p'r'$: 
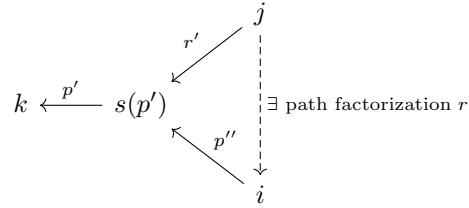
\[\begin{tikzcd}[ampersand replacement=\&]
	\&\& j \\
	k \& {s(p')} \\
	\&\& i
	\arrow["{r'}"', from=1-3, to=2-2]
	\arrow["{\exists \textup{ path factorization }r}", dashed, from=1-3, to=3-3]
	\arrow["{p'}"', from=2-2, to=2-1]
	\arrow["{p''}"', from=3-3, to=2-2]
\end{tikzcd}\]
Moreover, $r':j\rightarrow s(p')$ is a path with $j=\max(r')>k$ and $p'':i\rightarrow k'$ is a path with $\max(p'')=k'=t(p'')$, so that by induction hypothesis there is a path $r$ such that $r'=p''r$, and hence $q=p'r'=p'p''r=pr$.\\
Now, let us show \eqref{condition2}. Suppose $q:j\rightarrow i$ is any path in $Q$ with $j>i$ and assume $p, p'\in E_i$ with $t(p), t(p')< j$. Making $q$ shorter if necessary, we can assume that $\max(q)=j$ and $j>t(p), t(p')$. Then $(p, pq), (p', p'q)\in E_{ij}'$, so that by Corollary \ref{corollary_criterion}, $p=p'$.\\
On the other hand, suppose that \eqref{condition1} and \eqref{condition2} hold, that is, suppose that
\begin{enumerate}
        \item for all paths $p:i\rightarrow k$ and $q:j\rightarrow k$ with $\max(p)=k>i$ and $j>k$ there is a path $r:j\rightarrow i$ such that $q=pr$; and
        \item for any path $q:j\rightarrow i$ with $j>i$ there is at most one right-minimal direction preserving path $p$ starting in $i$ such that $t(p)< j$.
    \end{enumerate}
    Suppose $(p,q)\in E_{ij}'$ with $q$ non-trivial. Then by \eqref{condition1},  $q=pq''$ for some $q''$, and, by \eqref{condition2} we have that for any $p'\in E_i\setminus \{p_i\}$, $t(p')> j$, so that $\max(p'q'')>j$.
\end{proof}
The above criterion is much easier to check. Moreover, it is relatively easy to see how it behaves under certain modifications. For example, one can directly see the following:
\begin{remark}\label{remark_subquiver}
    Suppose $Q'$ is a subquiver of $Q$. If $Q$ fulfills the criteria from the above proposition, then so does $Q'$. Therefore, if $A=\field Q$ admits a regular exact Borel subalgebra, so does $A':=\field Q'$. In particular, if $A$ admits a regular exact Borel subalgebra, so does $A/AeA$, where $e=\sum_{i\in S}e_i$ for some $S\subseteq \{1, \dots, n\}$.\\
    A similar result was obtained in a much more general setting by Conde and Külshammer in \cite[Theorem 5.9]{idempotents}. However, the statement here is not directly implied by \cite[Theorem 5.9]{idempotents}, which can be seen e.g. by considering the idempotent $e_1$ in Example \ref{example_regular_path}(1).
\end{remark}
We also observe the following:
\begin{corollary}\label{corollary_injectiveuniserial}
    Suppose $I=(0)$, and assume that $A$ admits a regular exact Borel subalgebra. Then $\nabla_i^A$ is injective or uniserial for all $1\leq i\leq n$.
\end{corollary}
\begin{proof}
   Let $1\leq k\leq n$, and suppose that $\nabla_k$ is not injective. Then, 
   there is a path $q: j\rightarrow k$ for some $j>k$.
   Then any non-trivial path $p:i\rightarrow k$ with $\max(p)=k$ is a subpath of $q$, since by the absence of loops $k>i$, so that we can apply Condition \ref{condition1}, Theorem \ref{proposition_regularity_hereditary}. Hence $\nabla_j$ is a submodule of the uniserial submodule $N_{q'}$ of  $\nabla_k$ generated by $(q')^{*}$, and thus uniserial.
\end{proof}
We end the section with some examples:
\begin{example}\label{example_regular_path}
\begin{enumerate}
 \item Let $A=\field Q$ where $Q$ is the quiver 
\[\begin{tikzcd}[ampersand replacement=\&]
	1 \& 2 \& 3.
	\arrow["\alpha", from=1-1, to=1-2]
	\arrow["\beta"', from=1-3, to=1-2]
\end{tikzcd}\]
Then it has an exact Borel subalgebra given by the subquiver
\[\begin{tikzcd}[ampersand replacement=\&]
	1 \& 2 \& 3
	\arrow["\alpha", from=1-1, to=1-2]
\end{tikzcd}\]
but this is not regular, since the non-direction-preserving path $\beta:3\rightarrow 2$ does not contain the right-minimal direction-preserving path $\alpha:1\rightarrow 2$. In fact, it follows from  \cite[Theorem 60]{Markus} that this $A$ does not admit a regular exact Borel subalgebra. However, note that $\nabla_1^A=L_1$, $\nabla_3^A=L_3$ and $\nabla_2^A=P_1^A$ are all uniserial, so that the opposite implication in Corollary \ref{corollary_injectiveuniserial} does not hold.
\item   Let $A=\field Q$ where $Q$ is the quiver 
\[\begin{tikzcd}[ampersand replacement=\&]
	1 \& 2 \& 3
	\arrow["\alpha"', from=1-2, to=1-1]
	\arrow["\beta", from=1-2, to=1-3]
\end{tikzcd}\]
Then it has an exact Borel subalgebra given by the subquiver
\[\begin{tikzcd}[ampersand replacement=\&]
	1 \& 2 \& 3
	\arrow["\beta", from=1-2, to=1-3]
\end{tikzcd}\]
This is regular, since it has only one right-minimal direction-preserving path, namely $\beta$ and only one non-direction-preserving path, namely $\alpha$, and since they don't terminate in the same vertex, there is nothing to be checked. Again, this example is encompassed in \cite[Theorem 60]{Markus}.
    \item Let $A=\field Q$ where $Q$ is the quiver 
\[\begin{tikzcd}[ampersand replacement=\&]
	\& 1 \\
	4 \&\& 3 \\
	\& 2
	\arrow["\delta", from=1-2, to=2-3]
	\arrow["\alpha", from=2-1, to=1-2]
	\arrow["\beta"', from=2-1, to=3-2]
	\arrow["\gamma"', from=3-2, to=2-3]
\end{tikzcd}\]
Then, it has an exact Borel subalgebra given by the subquiver
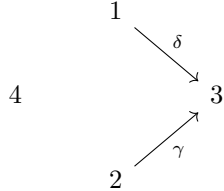
\[\begin{tikzcd}[ampersand replacement=\&]
	\& 1 \\
	4 \&\& 3 \\
	\& 2
	\arrow["\delta", from=1-2, to=2-3]
	\arrow["\gamma"', from=3-2, to=2-3]
\end{tikzcd}\]
but this is not regular, since the non-direction-preserving path $\delta\alpha:4\rightarrow 3$ does not contain the right-minimal direction-preserving path $\gamma:2\rightarrow 3$.
\end{enumerate}
\end{example}
\section{Further Examples}\label{section6}
As in the previous section, we assume that $\field$ is algebraically closed throughout this section.\\
In \cite{idempotents}, Conde and Külshammer investigated exact Borel subalgebras of idempotent subalgebras and quotient algebras.
In the following, given a monomial quasi-hereditary algebra, we explain how the exact Borel subalgebra of $A$ generated by right minimal paths relates to the exact Borel subalgebras generated by right-minimal paths of certain quotient algebras and certain idempotent subalgebras of $A$. More concretely, we show that for a quasi-hereditary monomial algebra $A$, if $B_{\min}$ is the exact Borel subalgebra of $A$ generated by right-minimal direction-preserving paths, then for $e=\sum_{i=i_0}^n e_i$, $eB_{\min}e$ is the exact Borel subalgebra of $eAe$ generated by right-minimal direction-preserving paths. Note that by \cite[Theorem 5.9]{idempotents}, it is already known that $eB_{\min}e$ is an exact Borel subalgebra.\\
Moreover, we show that if $J$ is any monomial ideal in $A$ such that $A/J$ is quasi-hereditary with the induced order, then $B_{\min}/(B_{\min}\cap J)$ is the exact Borel subalgebra of $A/J$ generated by right-minimal direction-preserving paths.
Additionally, we show that if $A$ is hereditary and admits a regular exact Borel subalgebra, and $J$ is as above, then $A/J$ admits a regular exact Borel subalgebra.
\begin{proposition}\label{proposition_Borel_eAe}
   Let $1\leq i_0\leq n$, $e:=\sum_{i=i_0}^n e_i^A$, and let $B$ be the exact Borel subalgebra of $A$ generated by the right-minimal direction preserving paths. Then $eBe$ is the exact Borel subalgebra of $eAe$ generated by the right-minimal direction preserving paths.
\end{proposition}
\begin{proof}
    Let $p+I$ be a right-minimal direction-preserving path in $A$, such that $p+I=e(p+I)e$. Then $p+I$ is clearly a right-minimal direction preserving path in $eAe$.\\
    On the other hand, let $p+I$ be a direction preserving path in $eAe$. Then $p+I$ is also direction-preserving in $A$. Suppose $p+I=(p'+I)(p''+I)$ in $A$ where $p''+I$ is direction-preserving in $A$. Then $s(p')=t(p'')\geq s(p'')=s(p)\geq i_0$, so that $p''+I=e(p''+I)e$ and $p'+I=e(p'+I)e$, and hence $p+I$ is not right-minimal direction-preserving in $eAe$.
    Hence the right-minimal direction preserving paths in $eAe$ are exactly the right-minimal direction preserving paths in $A$ which lie in $eAe$.
    In particular, the subalgebra $B'$ of $eAe$ which is generated by the right-minimal direction-preserving paths in $eAe$ lies in $eBe$.\\
    Moreover, if $b\in eBe$ then we can write $b$ as a sum of products of right-minimal direction-preserving paths in $A$, say $b=\sum_{k=1}^Kp_k+I=\sum_{k=1}^K \prod_{l=1}^{L_k} p_{lk}+I=\sum_{k=1}^K (p_{1k}+I)\cdot \dots \cdot (p_{L_kk}+I)$. We need to show that $p_{lk}+I\in eAe$ for all $1\leq k\leq K$, $1\leq l\leq L_k$. Let us fix $1\leq k\leq K$, and proceed by induction on $L_k-l$. If $l=L_k$ then, since $b=be$ and $A$ is monomial, $p_k=p_ke$ so that $s(p_{k,L_k})=s(p_k)\geq i_0$. Since $p_{k,L_k}$ is direction-preserving this implies that $t(p_{k,L_k})\geq i_0$, so that $p_{k,L_k}+I=e(p_{k,L_k}+I)e$. Now suppose $l<L_k$. By induction hypothesis, $p_{k,l+1}+I=e(p_{k,l+1}+I)e$, so that $t(p_{k,l+1})=s(p_{k,l})\geq i_0$. Since $p_{k,l}$ is direction-preserving this implies that $t(p_{k,l})\geq i_0$, so that $p_{k,l}+I=e(p_{k,l}+I)e$.
\end{proof}
The following example shows that this is in general not true if $e=\sum_{i\in S}e_i$ for some subset $S\subseteq \{1, \dots, n\}$ that is not of the form $S=\{i_0, i_0+1, \dots, n\}$.
\begin{example}
    Let $Q$ be the quiver 
\[\begin{tikzcd}[ampersand replacement=\&]
	1 \& 3 \& 2
	\arrow["\alpha", from=1-1, to=1-2]
	\arrow["\beta", from=1-2, to=1-3]
\end{tikzcd}\]
and $A=\field Q$. Then the exact Borel subalgebra $B_{\min}$ of $A$ generated by the right-minimal direction preserving paths is given by the subquiver
\[\begin{tikzcd}[ampersand replacement=\&]
	1 \& 3 \& 2
	\arrow["\alpha", from=1-1, to=1-2]
\end{tikzcd}\]
Let $e:=e_1+e_2$. Then we can identify $eAe$ with the path algebra of the quiver
\[\begin{tikzcd}[ampersand replacement=\&]
	1 \&\& 2
	\arrow["{\beta\alpha}", from=1-1, to=1-3]
\end{tikzcd}\]
and $eBe$ with the subalgebra given by the subquiver 
\[\begin{tikzcd}[ampersand replacement=\&]
	1 \&\& 2
\end{tikzcd}\]
This is not an exact Borel subalgebra of $eAe$.
\end{example}
\begin{proposition}
    Let $J$ be a monomial ideal in $A$ (possibly containing trivial paths) such that $A/J$ is quasi-hereditary with the induced order, and let $B$ be the exact Borel subalgebra of $A$ generated by the right-minimal direction preserving paths. Then $B/(B\cap J)$ is the exact Borel subalgebra of $A/J$ generated by the right-minimal direction-preserving paths.
\end{proposition}
\begin{proof}
    This is clear, since a path $p+J\in A/J$ is right-minimal direction preserving in $A/J$ if and only if $p\notin J$ and $p+I$ is right-minimal direction preserving in $A$.
\end{proof}
Quotienting by monomial ideals is also compatible with regularity of the exact Borel subalgebra $B_{\min}$, provided that $A$ is hereditary.  
\begin{proposition}
    Suppose $I=(0)$ and assume that $A$ admits a regular exact Borel subalgebra. Let $J$ be a monomial ideal in $A$ such that $A/J$ is quasi-hereditary with the induced order. Then $A/J$ admits a regular exact Borel subalgebra.
\end{proposition}
\begin{proof}
Let $p_1, \dots, p_n$ be the monomial generators of $J$ and let $J_0:=(p_i|l(p_i)\leq 1)$. Then $A/J_0$ is isomorphic to the path algebra of some subquiver of $Q'$, so that it is hereditary and admits a regular exact Borel subalgebra by Remark \ref{remark_subquiver}. Moreover, $J/J_0$ is a monomial ideal in $A/J_0$ generated by paths of length at least two. By replacing $A$ with $A/J_0$, we can therefore assume without loss of generality that $J\subseteq \rad(A)^2$.\\
Let $p+J$ be a right-minimal direction preserving path in $A/J$, and $q+J$ be a path in $A/J$ such that $s(q)>t(q)=t(p)$, and such that for any path $q'+J$ in $A/J$ with $\max(q')\leq s(q)$ and $q'p\in J$ we have $q'q\in J$.
We have to show that $q+J=pq''+J$ for some path $q''+J\in A/J$ with $p'q''\in J$ for all right-minimal direction preserving paths $p'+J\neq p+J$ with $s(p')=s(p)$ and $\max(p'q'')=j$.\\
We have seen that $p$ is a right-minimal direction preserving path in $A$. Moreover, $q$ is a path in $A$ such that $s(q)>t(q)=t(p)$. Hence by Proposition \ref{proposition_regularity_hereditary}, there is a path $q''$ in $A$ such that $q=pq''$. Moreover, since $s(q'')=s(q)>t(q)$, we can conclude, again by  Proposition \ref{proposition_regularity_hereditary}, that there is only one right-minimal direction preserving path $p'=p$ in $A$ with $s(p')=s(p)$ and $t(p')\leq s(q)$. Hence, there is no right-minimal direction preserving path $p'+J\neq p+J$ with $s(p')=s(p)$ and $t(p')\leq s(q)$ in $A/J$, so that, by that absence of loops, $\max(p'q'')>s(q)=j$ for all right direction preserving path $p'+J\neq p+J$ with $s(p')=s(p)$.
\end{proof}
Unfortunately, this does not work if $A$ is not hereditary, as can be seen in the following example:
\begin{example}
Let $Q$ be the quiver 
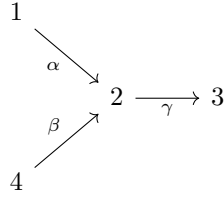
\[\begin{tikzcd}[ampersand replacement=\&]
	1 \\
	\& 2 \& 3 \\
	4
	\arrow["\alpha"', from=1-1, to=2-2]
	\arrow["\gamma"', from=2-2, to=2-3]
	\arrow["\beta", from=3-1, to=2-2]
\end{tikzcd}\]
and let $I=(\gamma\alpha)$. Then $A=\field Q/I$ is quasi-hereditary by \cite[Theorem 3.10]{schroll}, and admits a regular exact Borel subalgebra by Lemma \ref{lemma_regularity}, since
\begin{align*}
    E_{12}'=\{(\alpha+I, e_2^A)\}\\
    E_{23}'=\{(\gamma+I, e_3^A)\}\\
    E_{24}'=\{(\gamma+I, \gamma\beta+I)\}\\
    E_{ij}'=\emptyset \textup{ else}.
\end{align*}
Note that $(\alpha+ I, \beta+I)\notin E_{14}'$ since $\gamma\alpha\in I$ and $\gamma\beta\notin I$.\\
Let $J:=(\gamma+I)$ be the ideal generated by $\gamma+I$ in $A$. Then $A/J$ is isomorphic to the path algebra of the quiver $Q'$ given by 
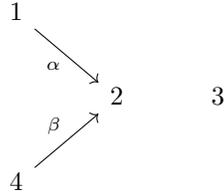
\[\begin{tikzcd}[ampersand replacement=\&]
	1 \\
	\& 2 \& 3 \\
	4
	\arrow["\alpha"', from=1-1, to=2-2]
	\arrow["\beta", from=3-1, to=2-2]
\end{tikzcd}\]
which does not admit a regular exact Borel subalgebra by Proposition \ref{proposition_regularity_hereditary}, since $\beta$ does not factor through $\alpha$.\\
Similarly, if $J':=(e_3^A)$ is the ideal generated by $e_3^A$ in $A$, then $A/J'$ is isomorphic to to the path algebra of the quiver $Q''$ given by
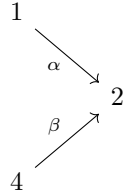
\[\begin{tikzcd}[ampersand replacement=\&]
	1 \\
	\& 2 \\
	4
	\arrow["\alpha"', from=1-1, to=2-2]
	\arrow["\beta", from=3-1, to=2-2]
\end{tikzcd}\] which does not admit a regular exact Borel subalgebra for the same reason, see also Example \ref{example_regular_path}.
\end{example}
\section{Counting quasi-hereditary structures with regular exact Borels}\label{section7}
Suppose $A$ is the path algebra of a quiver with underlying graph $\textup{A}_n$. In this setting \cite[Theorem 60]{Markus} characterizes the partial orders with respect to which $A$ admits a regular exact Borel subalgebra.
In the following, we will apply our result to re-prove this statement, as well as to obtain similar results for type $\textup{D}$ and $\textup{E}$. Moreover, using similar techniques as in \cite{combinatorics}, where quasi-hereditary structures for type $\textup{D}$ and $\textup{E}$ were counted, we will count the quasi-hereditary structures for type  $\textup{D}$ and $\textup{E}$ which admit a regular exact Borel subalgebra. Here,  as in \cite{combinatorics}, we only address the case where those cannot be deconcatenated at a sink or a source in the sense of \cite{combinatorics}. In view of \cite[Theorem 1.3]{combinatorics} as well as \cite[Proposition 54, 55 and 57]{Markus}, iterated deconcatenation may be used to deduce similar results in the cases where there are sinks and/or sources.\\
So far, we have worked with a particular choice of a total order on the vertex set of our graph, corresponding to a total order on the isomorphism classes of simple modules $\Sim(A)$. In the following, we will need to more generally consider partial orders on $\Sim(A)$. 
Therefore, let us recall some important definitions in this context:
\begin{definition}
     Let $A$ be a finite-dimensional algebra and let $\leq$ be a partial order on $\Sim(A)$. Then $\leq$ is called adapted if for any module $M\in \modu A$ with simple socle $L'$ and simple top $L$ such that $L$ and $L'$ are not comparable, there is a composition factor $L''$ of $M$ such that $L''>L$ and $L''>L'$.
\end{definition}
\begin{remark}\label{remark_adapted_quiver}
    Let $\leq$ be a partial order on $\Sim(\field Q)$ for a quiver $Q$. Then, $\leq$ is adapted if and only if for any path $p:i\rightarrow j$ between non-comparable vertices $i$ and $j$, there is an inner vertex $k$ of $p$ such that $k>i$ and $k>j$.
\end{remark}
\begin{definition}
    Let $A$ be a finite-dimensional algebra. Two partial orders on $\Sim(A)$ are called equivalent if they give rise to the same standard and costandard modules. 
\end{definition}
Let us state some well-known facts about quasi-hereditary algebras. Most of them are easy to check, nevertheless, we provide references for the non-trivial ones:
\begin{itemize}\label{list_of_properties}
    \item  Since the definition of quasi-heredity only references the standard modules, and not the partial order, it is clear that if $\leq$ is equivalent to $\leq'$ then $(A, \leq)$ is quasi-hereditary if and only if $(A, \leq')$ is.
    \item  \cite[p. 3-4]{DlabRingel} If $\leq$ is an adapted partial order and $\leq'$ is a refinement of $\leq$, then $\leq$ and $\leq'$ are equivalent. In particular, any adapted partial order is equivalent to a total order.
    \item \cite[p.3]{coulembier} If $\leq$ is an adapted partial order, then there is a unique adapted partial order $\leq_{\ess}$ such that any partial order which is equivalent to $\leq$ is a refinement of $\leq_{\ess}$. This partial order is called the essential partial order of the equivalence class of $\leq$.
    \item  Since every refinement of an adapted partial order is adapted, this implies that if $\leq$ and $\leq'$ are equivalent partial orders and $\leq$ is adapted, then so is $\leq'$.
    \item  \cite[p. 3]{coulembier} Two adapted partial orders give rise to the same costandard modules if and only if they give rise to the same standard modules.
    \item \cite[Proposition 1.4.12]{thesis} If $\leq$ is a partial order such that $(A, \leq)$ is quasi-hereditary, then $\leq$ is adapted.
    \item \cite{CPS} If  $A=\field Q$ is the path algebra of a quiver $Q=(V, E)$, then $(A, \leq)$ is quasi-hereditary for any adapted partial order.
\end{itemize}
 Recall that $\Sim(\field Q)$ is in bijection with the vertex set $V$ of $Q$, so that we identify partial orders on $\Sim(\field Q)$ with partial orders on $V$. In particular, by quasi-hereditary structures on $\field Q$ we mean equivalence classes of partial orders on $V$ which correspond to equivalence classes of adapted partial orders on $\Sim(\field Q)$.\\
 In the previous sections, we assumed that the vertices of $Q$ were indexed by numbers $1, \dots, n$ and the order we considered was the natural order. In the following, this is impractical, and we will instead consider arbitrary adapted partial orders. Recall that if we have some total order on $V$, we can label the vertices in $V$ by $1, \dots, n$, so that the given total order corresponds to the natural order. In this way, we can still use our results from the previous sections for any total order on $V$, reinterpreting the symbols $\leq$, $\max$ and $\min$ according to our chosen total order.
\begin{proposition}\cite[Theorem 60]{Markus}
    Let $Q$ be a quiver with underlying graph $\textup{A}_n$ and let $A:=\field Q$. Moreover, let $\leq_A$ be an essential adapted partial order on $\Sim(A)$. Then $(A, \leq_A)$ admits a regular exact Borel subalgebra if and only if any vertex $k$ in $Q$ with two distinct ingoing arrows $\alpha: i\rightarrow k$ and $\beta: j \rightarrow k$ is either maximal or minimal with respect to $\leq_A$.
\end{proposition}
\begin{proof}
    Suppose $(A, \leq_A)$ admits a regular exact Borel subalgebra and $k$ is a vertex in $Q$ with two distinct ingoing arrows $\alpha: i\rightarrow k$ and $\beta: j \rightarrow k$.  Recall that  since $\leq_A$ is an essential adapted partial order, \cite[Definition 1.2.5]{coulembier} and \cite[Lemma 2.5]{DlabRingel} imply that $k$ is minimal with respect to $\leq_A$ if and only if $\Delta_k=L_k$ and $\nabla_k=L_k$, and $k$ is maximal with respect to $\leq_A$ if and only if $\Delta_k=P_k$ and $\nabla_k=I_k$.\\
    First, note that since $k$ has two ingoing arrows and $Q$ has underlying graph $\textup{A}_n$, $k$ has no outgoing arrows, so that $P_k=L_k=\Delta_k$. Thus it remains to be shown that either  $\nabla_k=I_k$ or $\nabla_k=L_k$.\\
    Let $\leq_A'$ be a refinement of $\leq_A$ to a total order.
    Then by Proposition \ref{proposition_regularity_hereditary}, Condition \eqref{condition1}, it is not possible that $i<_A' k<_A' j$ or $j<_A' k<_A' i$.\\
    Hence $k<_A' i, j$ or $k>_A' i, j$. If $k<_A' i, j$, then,  since $\leq_A'$ gives rise to the same standard and costandard modules as $\leq_A$, we have that $\nabla_k=L_k$. So let us assume that $k>_A' i, j$. Then, by Proposition \ref{proposition_regularity_hereditary}, Condition \eqref{condition1}, there is no path $p: i'\rightarrow k$ where $i'>_A' k$,  since such a path would have to factor both through $\alpha$ and $\beta$, which is impossible. This implies that $\nabla_k=I_k$.\\
    Before we show the other direction, let us note that any two paths $p:i\rightarrow j$ and $p':i\rightarrow j'$ that start in the same vertex either start with a different arrow, or one is a right subpath of the other. To see this, suppose $q:i\rightarrow k$ is the maximal common right subpath of $p$ and $p'$, i.e. there are $r:k\rightarrow j$ and $r':k\rightarrow j'$ such that $p=rq$, $p'=r'q$ and $q$ is of maximal length with respect to this condition. Then, by maximality of $q$, $r$ and $r'$ do not start with the same arrow. However, $Q$ is of type $\textup{A}_n$, so the vertex $k$ has degree at most two. Hence one of $q$, $r$ and $r'$ must be trivial. If $q$ is trivial, then $p$ and $p'$ start with a different arrow, if, on the other hand, $r$ is trivial, then $p=q$ is a right subpath of $p'$ and if $r'$ is trivial, then $p'=q$ is a right subpath of $p$.\\
    Dually, any two paths $p:i\rightarrow j$, $p':i'\rightarrow j$ either end with a different arrow or one is a left subpath of the other.\\
    In particular, this implies that Condition \eqref{condition2} from Theorem \ref{proposition_regularity_hereditary} is automatically fulfilled for $Q$, since if $p:i\rightarrow k$ and $p':i\rightarrow k'$ are two right-minimal direction preserving paths starting in the same vertex, then by right-minimality, neither can be a right subpath of the other, so that they must start with different arrows, and therefore $i$ must have two distinct outgoing arrows. But then there can be no path $q:j\rightarrow i$ ending in $i$ with $j>i$, since the degree of $i$ is not greater than two.\\
   Let us suppose that $(A, \leq_A)$ does not admit a regular exact Borel subalgebra. Then either Condition \eqref{condition1} or Condition \eqref{condition2} from Theorem \ref{proposition_regularity_hereditary} fails for $Q$. We have seen that Condition \eqref{condition2} is always fulfilled, so that Condition \eqref{condition1} must fail. Therefore, there are paths  $p:i\rightarrow k$ and $q:j\rightarrow k$ in $Q$ with $\max(p)=k>_Ai$ and $j>_Ak$ such that there is no $r:j\rightarrow i$ with $q=pr$. In particular, $p$ and $q$ end in the same vertex, and $p$ is not a left subpath of $q$. But $q$ is also not a left subpath of $p$, since $\max(q)=j>k=\max(p)$. By what we have seen above, $q$ and $p$ must therefore end in a different arrow, so that $k=t(p)=t(q)$ has two distinct ingoing arrows. Moreover, by assumption $i<_Ak<_Aj$, so that $k$ is neither maximal nor minimal with respect to $\leq_A$.
\end{proof}
 In the following, denote by $C_n:=\frac{1}{n+1}\begin{pmatrix}2n \\ n\end{pmatrix}$ the $n$-th Catalan number and for all $n_a, n_b, n_c\geq 0$ let $Q(n_a, n_b, n_c)$ be the quiver
\[\begin{tikzcd}[ampersand replacement=\&]
	\&\&\&\& {b_1} \& \dots \& {b_{n_b}} \\
	{a_{n_a}} \& \dots \& {a_1} \& {a_0} \\
	\&\&\&\& {c_1} \& \dots \& {c_{n_c}.}
	\arrow[from=1-5, to=1-6]
	\arrow[from=1-6, to=1-7]
	\arrow[from=2-1, to=2-2]
	\arrow[from=2-2, to=2-3]
	\arrow[from=2-3, to=2-4]
	\arrow[from=2-4, to=1-5]
	\arrow[from=2-4, to=3-5]
	\arrow[from=3-5, to=3-6]
	\arrow[from=3-6, to=3-7]
\end{tikzcd}\]
If we want to count quasi-hereditary structures admitting regular exact Borel subalgebras on path algebras of quivers of type $\textup{D}$ and $\textup{E}$, we can use \cite[Proposition 54, 55 and 56]{Markus} as well as symmetry in $b$ and $c$ to reduce to the case of the following quivers:
\begin{center}
\begin{tabular}{c|c}
    Type $\textup{D}_n$ & $Q(n-3, 1, 1), Q(n-3, 1, 1)^{\op}, Q(1, n-3, 1), Q(1, n-3, 1)^{\op}$ \\
    Type $\textup{E}_6$ & $Q(1, 2, 2), Q(1, 2, 2)^{\op}, Q(2,1, 1), Q(2, 1,1)^{\op}$ \\
    Type $\textup{E}_7$ & $Q(1, 3, 2), Q(1, 3, 2)^{\op}, Q(2, 3, 1), Q(2, 3, 1)^{\op}, Q(3, 2, 1), Q(3, 2,1)^{\op}$ \\
    Type $\textup{E}_8$ & $Q(1, 4, 2), Q(1, 4, 2)^{\op}, Q(2, 4, 1), Q(2, 4, 1)^{\op}, Q(4, 2, 1), Q(4, 2,1)^{\op}.$ \\
\end{tabular}
\end{center}
For more details on this, see \cite[Example 5.10]{combinatorics}.\\

Denote by $Q_b(n_b)$ the subquiver 
\[\begin{tikzcd}[ampersand replacement=\&]
	{b_{1}} \& {b_2} \& \dots \& {b_{n_b}}
	\arrow[from=1-1, to=1-2]
	\arrow[from=1-2, to=1-3]
	\arrow[from=1-3, to=1-4]
\end{tikzcd}\] and by $Q_c(n_c)$ the subquiver 
\[\begin{tikzcd}[ampersand replacement=\&]
	{c_{1}} \& {c_2} \& \dots \& {c_{n_c}}
	\arrow[from=1-1, to=1-2]
	\arrow[from=1-2, to=1-3]
	\arrow[from=1-3, to=1-4]
\end{tikzcd}\]
of $Q(n_a, n_b, n_c)$.
Moreover, denote for $1\leq k\leq n_b$ by $Q_{b, k}$ the subquiver
\[\begin{tikzcd}[ampersand replacement=\&]
	{b_k} \& {b_{k+1}} \& \dots \& {b_{n_b}}
	\arrow[from=1-1, to=1-2]
	\arrow[from=1-2, to=1-3]
	\arrow[from=1-3, to=1-4]
\end{tikzcd}\]
and for $1\leq k\leq n_c$ denote by $Q_{c, k}$ the subquiver
\[\begin{tikzcd}[ampersand replacement=\&]
	{c_k} \& {c_{k+1}} \& \dots \& {c_{n_c}}
	\arrow[from=1-1, to=1-2]
	\arrow[from=1-2, to=1-3]
	\arrow[from=1-3, to=1-4]
\end{tikzcd}\]
of $Q(n_a, n_b, n_c)$.
In the following, we will often want to pass from the quivers $Q(n_a, n_b, n_c)$ and $Q(n_a, n_b. n_c)^{\op}$ to subquivers such as the ones above. Therefore, let us define the following:
\begin{definition}
    Let $Q$ be any quiver, $Q'$ a subquiver, and $\leq$ a partial order on the vertices of $Q$. Then we denote by $\leq_{Q'}$ the restriction of this partial order to $Q'$.
\end{definition}
\begin{remark}\label{remark_partial_order_restriction}
    It is easy to see that  if $\leq$ is an adapted partial order for $\field Q$, then  $\leq_{Q'}$ is an adapted partial order for $\field Q'$ for any subquiver $Q'$.
    Moreover, by \cite[p. 3]{coulembier}, if $\leq$ is an adapted partial order equivalent to $\hat{\leq}$, then they are both refinements of the corresponding essential partial order $\leq_{\ess}$, so that both $\leq_{Q'}$ and $\hat{\leq}_{Q'}$ are refinements of  $(\leq_{\ess})_{Q'}$. Since $(\leq_{\ess})_{Q'}$ is adapted, this implies that $\leq_{Q'}$ and $\hat{\leq}_{Q'}$ are both equivalent to $(\leq_{\ess})_{Q'}$, and hence equivalent to each other.
\end{remark}
Sometimes, we will also want to pass from partial orders on the vertex sets of subquivers to partial orders on the vertex set of the big quiver. The following definition describes how exactly we intend to do this.
\begin{definition}
    Let $Q=(V, E)$ be a quiver and $Q'=(V', E')$ and $Q''=(V'', E'')$ be two subquivers such that $V=V'\cup V''$ and $V'\cap V''=\emptyset$. Then for any two partial orders $\leq'$ on $V'$ and $\leq''$ on $V''$ we define an induced partial order $\leq$ via
    \begin{align*}
        v\leq w :\Leftrightarrow &(v,w\in V' \textup{ and }v\leq'w) \\\textup{ or } &(v,w\in V'' \textup{ and }v\leq''w)\\ \textup{ or } &(v\in V'\textup{ and }w\in V'' )
    \end{align*}
    Moreover, we let $\iota_{Q', Q''}$ be the map from the set of pairs consisting of a partial order on $V'$ and a partial order on $V''$ to the set of partial orders on $V$ which takes a pair $(\leq', \leq'')$ to the induced partial order $\leq$.\\
    If we instead divide into three subquivers $Q', Q''$ and $Q'''$ we write $\iota_{Q', Q'', Q'''}$ for $\iota_{Q'\cup Q'', Q'''}\circ (\iota_{Q', Q''}\times \id)=\iota_{Q', Q'' \cup Q'''}\circ (\id \times \iota_{Q'', Q'''})$. 
\end{definition}
\begin{remark}\label{remark_adapted_restriction}
    Using Remark \ref{remark_adapted_quiver}, it is easy to see that if there are no edges from $V'$ to $V''$ or no edges from $V''$ to $V'$, then if $\leq'$ and $\leq''$ are adapted, so is $\leq$. Correspondingly, if we divide into three quivers $Q'$, $Q''$ and $Q'''$ and there are for example no edges from $V'$ to $V''$, from $V'''$ to $V'$ and from $V'''$ to $V''$; or no edges from  $V''$ to $V'$, from $V''$ to $V'''$ and from $V'$ to $V'''$, then if $\leq'$, $\leq''$ and $\leq'''$ are adapted, then so is $\leq$.\\
    In general, two adapted partial orders $\leq'$ and $\leq''$ do not necessarily induce an adapted order. For example, for the $\textup{A}_3$ quiver
\[\begin{tikzcd}[ampersand replacement=\&]
	a \& b \& c
	\arrow[from=1-1, to=1-2]
	\arrow[from=1-2, to=1-3]
\end{tikzcd}\]
picking $V'=\{b\}$ and $V''=\{a, c\}$ and the partial orders $\leq'$ and $\leq''$ as the trivial orders on $V'$ respectively $V''$ yields as the induced order $\leq$ on $V$ the trivial order. In this case, $\leq'$ and $\leq''$ are adapted, since $\field Q'$ and $\field Q''$ are semisimple, but $\leq$ is not adapted.
\end{remark}
\begin{lemma}\label{lemma_special_quiver_regular}
   Suppose $\leq$ is a total order on the vertex set of $Q(n_a, n_b, n_c)$. Then the algebra $\field Q(n_a, n_b, n_c)$ admits a regular exact Borel subalgebra with respect to $\leq$ if and only if one of the following holds:
   \begin{itemize}
       \item there is $1\leq k\leq n_b+1$ such that $a_0$ is maximal in the subquiver $Q(0, k-1, 0)$ and, if $k\leq n_b$ then $b_k>a_j$ for every $0\leq j\leq n_a$,  or
       \item there is $1\leq k\leq n_c+1$ such that $a_0$ is maximal in the subquiver $Q(0, 0, k-1)$ and, if $k\leq n_c$ then $c_k>a_j$ for every $0\leq j\leq n_a$.
   \end{itemize}
        Moreover,  $\field Q(n_a, n_b, n_c)^{\op}$ admits a regular exact Borel subalgebra with respect to $\leq$ if and only if
        one of the following holds:
        \begin{itemize}
        \item  $a_0>b_i$ for all $1\leq i\leq n_b$, or\\
        \item  $a_0>c_i$ for all $1\leq i\leq n_c$.
        \end{itemize}
\end{lemma}
\begin{proof}
    Recall that by Proposition \ref{proposition_regularity_hereditary}, $\field Q(n_a, n_b, n_c)$ resp. $\field Q(n_a, n_b, n_c)^{\op}$ admits a regular exact Borel subalgebra if and only if 
     Conditions \eqref{condition1} and \eqref{condition2} hold where we replace the natural order by our given total order.\\
     Note that for $\field Q(n_a, n_b, n_c)$,
    Condition \eqref{condition1} is  always fulfilled trivially, since the injective modules in $\field Q(n_a, n_b, n_c)$ are uniserial.\\
    Moreover, Condition \eqref{condition2} is always satisfied for $i=b_s$, $1\leq s\leq n_b$,  or $i=c_t$, $1\leq t\leq n_c$, since the corresponding projectives are uniserial, so there is at most one right-minimal direction preserving path starting in these vertices. Additionally, Condition \eqref{condition2} is trivially satisfied for  $i=a_{n_a}$ since there are no non-trivial paths ending in this vertex.\\
    Suppose there is $1\leq k\leq n_b+1$ such that  $a_0$ is maximal in the subquiver $Q(0, k-1, 0)$ and if $k\leq n_b$, $b_k>a_j$ for every $0\leq j\leq n_a$; and assume that Condition \eqref{condition2} is not fulfilled for some $a_i$, $0\leq i< n_a$. Then there is some $i<j\leq n_a$ such that $a_j>a_i$, and two distinct right-minimal direction preserving paths $p, p'$ starting in $a_i$ with endpoints less than $a_j$ in the total order. Since neither of these paths can factor through the other by right-minimality, one of them must end in $b_s$ for some $1\leq s\leq n_b$ and the other in $c_t$ for some $1\leq t\leq n_c$, so that $a_j>b_s>a_i$ and $a_j>c_t>a_i$, while $a_k<a_i$ for all $0\leq k< i$.
    In particular, $a_0\leq a_i<b_s$. If $k=n_b+1$, this is a contradiction, so let us assume that $k\leq n_b$. Then  $a_0\leq a_i<b_s$ implies that $s\geq k$ and since $b_k>a_j$ by assumption and $p$ is right-minimal direction preserving, we conclude that $s=k$. But this is a contradiction, since $a_j>b_s=b_k>a_j$.\\
    Analogously, one can show that  Condition \eqref{condition2} is fulfilled if there is  $1\leq k\leq n_c+1$ such that $a_0$ is maximal in the subquiver $Q(0, 0, k-1)$ and, if $k\leq n_c$ then $c_k>a_j$ for every $0\leq j\leq n_a$.\\
    Now suppose on the other hand that there is neither $1\leq k\leq n_b+1$ such that  $a_0$ is maximal in the subquiver $Q(0, k-1, 0)$ and if $k\leq n_b$, $b_k>a_j$ for every $0\leq j\leq n_a$; nor  $1\leq k\leq n_c+1$ such that $a_0$ is maximal in the subquiver $Q(0, 0, k-1)$ and, if $k\leq n_c$ then $c_k>a_j$ for every $0\leq j\leq n_a$.\\
    Then in particular, $a_0$ is not maximal in either $Q(0, n_b, 0)$ nor $Q(0, 0, n_c)$, since otherwise we could choose $k=n_b+1$ respectively $k=n_c+1$. Let $s:=\min\{k: b_k>a_0\}$ and $t:=\min\{k: c_k>a_0\}$, so that $s\leq n_b$, $t\leq n_c$. Then the paths $p:a_0\rightarrow b_s$ and $p':a_0\rightarrow c_t$ are right-minimal direction-preserving. Moreover, by assumption  there is $a_i>b_s$ and there is $a_j>c_t$. Let $a_l$ be maximal in $Q(n_a, 0, 0)$ and $q:a_j\rightarrow a_0$. Then $a_j>b_s, c_t$, so that \eqref{condition2} fails at the vertex $a_0$.\\
    Similarly, for  $\field Q(n_a, n_b, n_c)^{\op}$, Condition \eqref{condition2} is  always fulfilled trivially, since the projective modules in $\field Q(n_a, n_b, n_c)^{\op}$ are uniserial.\\
    Moreover, Condition \eqref{condition1} is always satisfied for $k=b_s$, $1\leq s\leq n_b$,  or $k=c_t$, $1\leq t\leq n_c$, since the corresponding injective modules are uniserial, so for any two paths ending in $k$, one must factor through the other.
    Suppose Condition \eqref{condition1} is not fulfilled for some $a_i$, $0\leq i\leq n_a$. Then there are paths $p, q\in Q(n_a, n_b, n_c)^{\op}$ ending in $a_i$ such that $q$ does not factor through $p$ and such that $\max(p)=s(p)$ and $s(q)>s(p)$. Note that $p$ does also not factor through $q$ since $\max(p)=s(p)<s(q)$. Hence one of them must start in $b_s$ for some $1\leq s\leq n_b$ and the other in $c_t$ for some $1\leq t\leq n_c$. In particular, $p$ must pass through $a_0$ and $a_0\neq s(p)$, so that $a_0< s(p)<s(q)$ and thus $a_0<b_s$ and $a_0<c_t$.
     On the other hand, suppose that  $a_0<b_s$ for some $1\leq s\leq n_b$ and $a_0<c_t$ for some $1\leq t\leq n_c$. Let us choose $s$ and $t$ minimal with respect to this. Then, if $p$ denotes the path from $b_s$ to $a_0$ and $q$ the path from $c_t$ to $a_0$, we have $\max(p)=b_s$ and $\max(q)=c_t$. Moreover, neither path factors through the other, and since $\leq$ is a total order, we have $b_s<c_t$ or $c_t<b_s$. Thus, Condition \eqref{condition1} fails at the vertex $a_0$.
    Hence $\field Q(n_a, n_b, n_c)^{\op}$ admits a regular exact Borel subalgebra if and only if
    $a_0>c_t$ for all $1\leq t\leq n_c$ or $a_0>b_s$ for all $1\leq s\leq n_b$.
\end{proof}
Since admitting a regular exact Borel subalgebra only depends on an equivalence class of partial orders, and not on the partial orders themselves, we immediately obtain that for two equivalent total orders, the above criteria are fulfilled by one if and only if they are fulfilled by the other.
However, we would like to extend this from total orders to arbitrary adapted partial orders. To do this, we will show that the conditions on the partial order in Lemma \ref{lemma_special_quiver_regular} are fulfilled for one representative of an equivalence class of partial orders if and only if they are fulfilled for any other representative, so that for a given class, they may be checked on any partial order instead of just on the total orders.
\begin{lemma}\label{lemma_equiv_orders}
   Let $\leq$ and $\hat{\leq}$ be two equivalent partial orders on
   a quiver $Q$ of directed type $\textup{A}_n$:
\[\begin{tikzcd}[ampersand replacement=\&]
	{x_1} \& {x_2} \& \dots \& {x_n}
	\arrow[from=1-1, to=1-2]
	\arrow[from=1-2, to=1-3]
	\arrow[from=1-3, to=1-4]
\end{tikzcd}\]
  Let $1\leq i\leq n$. Then, the following statements hold
    \begin{itemize}
        \item $x_i>x_j$ for all $i< j\leq n$ if and only if  $x_i\hat{>}x_j$ for all  $i< j\leq n$ .
        \item $x_i>x_j$ for all $1\leq j< i$ if and only if  $x_i\hat{>}x_j$ for all  $1\leq j< i$.
        \end{itemize}
\end{lemma}
\begin{proof}
    By symmetry in $\leq$ and $\hat{\leq}$, we only show one implication for each equivalence.\\
    Suppose  $x_i>x_j$ for all $i< j\leq n$. Then $\Delta_{x_i}=P_{x_i}$, so that  $x_i\hat{>}x_j$ for all $i<j\leq n$. Similarly, if  $x_i>x_j$ for all $1\leq j< i$, then $\nabla_{x_i}=I_{x_i}$, so that  $x_i\hat{>}x_j$ for all  $1\leq j< i$.
\end{proof}
Since restrictions of equivalent partial orders are equivalent, this shows in particular that the criteria in Lemma \ref{lemma_special_quiver_regular} are fulfilled for a given total order if and only if they are fulfilled for any equivalent adapted partial order. 
Since any adapted partial order can be refined to an equivalent total order, we thus obtain a version of Lemma \ref{lemma_special_quiver_regular} for adapted partial orders:
\begin{corollary}\label{corollary_special_quiver_regular}
   The algebra $\field Q(n_a, n_b, n_c)$ admits a regular exact Borel subalgebra with respect to a given adapted partial order $\leq$ if and only if one of the following holds:
   \begin{itemize}
       \item there is $1\leq k\leq n_b+1$ such that $a_0$ is maximal in the subquiver $Q(0, k-1, 0)$ and, if $k\leq n_b$ then $b_k>a_j$ for every $0\leq j\leq n_a$,  or
       \item there is $1\leq k\leq n_c+1$ such that $a_0$ is maximal in the subquiver $Q(0, 0, k-1)$ and, if $k\leq n_c$ then $c_k>a_j$ for every $0\leq j\leq n_a$.
   \end{itemize}
        Moreover,  $\field Q(n_a, n_b, n_c)^{\op}$ admits a regular exact Borel subalgebra with respect to a given partial order $\leq$ if and only if one of the following holds:
        \begin{itemize}
        \item  $a_0>b_i$ for all $1\leq i\leq n_b$, or\\
        \item  $a_0>c_i$ for all $1\leq i\leq n_c$.
        \end{itemize}
\end{corollary}
We would like to count the number of quasi-hereditary structures on the algebras $\field Q(n_a, n_b, n_c)$ and $\field Q(n_a, n_b, n_c)^{\op}$ that admit regular exact Borel subalgebras. Since a formula is known for the number of quasi-hereditary structures on quivers of type $\textup{A}_n$, the strategy is to divide $Q(n_a, n_b, n_c)$ into subquivers of type $\textup{A}_n$ using the criterion in Corollary \ref{corollary_special_quiver_regular}, and then count the number of quasi-hereditary structures on these subquivers. Let us begin by establishing some notation.
\begin{definition}
    For any quiver $Q$ denote by $\qh(Q)$ the set of quasi-hereditary structures on $\field Q$, and by $\qh^B(Q)$ the quasi-hereditary structures on $\field Q$ admitting a regular exact Borel subalgebra. Moreover, denote by $\qh_b(Q(n_a, n_b, n_c))$ the set  of quasi-hereditary structures on $\field Q(n_a, n_b, n_c)$ such that $a_0>b_i$ for all $1\leq i\leq n_b$ with respect to any/all corresponding partial orders, and by $\qh_c(Q(n_a, n_b, n_c))$ the set  of quasi-hereditary structures on $\field Q(n_a, n_b, n_c)$ such that $a_0>c_i$ for all \\$1\leq i\leq n_c$ with respect to any/all corresponding partial orders.
    Finally, denote for $1\leq k\leq n_b$ by $\qh_{b, k}(Q(n_a, n_b, n_c)$ the set of quasi-hereditary structures on  $\field Q(n_a, n_b, n_c)$ such that $a_0>b_i$ for all $1\leq i< k$ and $b_k>a_j$ for all $0\leq j\leq n_a$ with respect to any/all corresponding partial orders, and denote for for $1\leq k\leq n_c$ by $\qh_{c, k}(Q(n_a, n_b, n_c)$ the set of quasi-hereditary structures on  $\field Q(n_a, n_b, n_c)$ such that $a_0>c_i$ for all $1\leq i< k$ and $c_k>a_j$ for all $0\leq j\leq n_a$ with respect to any/all corresponding partial orders.\\
    Since there is exactly one partial order on the empty set, we use the convention that $\qh(\emptyset)=\qh_B(\emptyset)=\qh_b(\emptyset)=\qh_c(\emptyset)$ consists of this unique partial order, even though the empty quiver does not give rise to a path algebra.
\end{definition}
\begin{proposition}\label{proposition_bijections}
    There are bijections 
    \begin{align*}
       f_b: \qh_b(Q(n_a, n_b, n_c))&\rightarrow   \qh(Q_b(n_b))\times  \qh(Q(n_a, 0, n_c))\\
        [\leq]&\mapsto ([\leq_{Q_b(n_b)}], [\leq_{Q(n_a, 0, n_c)}])\\
         f_c:\qh_c(Q(n_a, n_b, n_c))&\rightarrow   \qh(Q_c(n_c))\times  \qh(Q(n_a, n_b, 0))\\
        [\leq]&\mapsto ([\leq_{Q_c(n_c)}], [\leq_{Q(n_a, n_b, 0)}]).
    \end{align*}
    with inverses given by 
    \begin{align*}
 g_b:  \qh(Q_b(n_b))\times  \qh(Q(n_a, 0, n_c))&\rightarrow  \qh_b(Q(n_a, n_b, n_c))\\
         ([\leq'], [\leq''])&\mapsto [\iota_{Q_b(n_b), Q( n_a, 0, n_c)}(\leq', \leq'')]\\
         g_c:  \qh(Q_c(n_c))\times  \qh(Q(n_a, n_b, 0))&\rightarrow  \qh_c(Q(n_a, n_b, n_c))\\
         ([\leq'], [\leq''])&\mapsto [\iota_{Q_c(n_c), Q( n_a, n_b, 0)}(\leq', \leq'')].
    \end{align*}
   Moreover, for every $1\leq k\leq n_b$ and $1\leq l \leq n_c$ we have bijections
    \begin{align*}
        f_{b, k}:\qh_{b, k}(Q(n_a, n_b, n_c))&\rightarrow \qh(Q_{b, k}(n_b)) \times \qh(Q_b(k-1))\times \qh(Q(n_a, 0, n_c) \\
        [\leq]&\mapsto ([\leq_{Q_{b, k}(n_b)}], [\leq_{Q_b(k-1)}], [\leq_{Q(n_a, 0, n_c)}])\\
        f_{c, l}:\qh_{c, l}(Q(n_a, n_b, n_c))&\rightarrow  \qh(Q_{c, l}(n_c)) \times \qh(Q_c(l-1))\times \qh(Q(n_a, n_b, 0)\\
        [\leq]&\mapsto ([\leq_{Q_{c, l}(n_c)}], [\leq_{Q_c(l-1)}], [\leq_{Q(n_a, n_b, 0)}])\\
    \end{align*}
    with inverses given by 
     \begin{align*}
 g_{b, k}: \qh(Q_{b, k}(n_b)) \times \qh(Q_b(k-1))\times &\qh(Q(n_a, 0, n_c) \rightarrow  \qh_{b, k}(Q(n_a, n_b, n_c))\\
         ([\leq'], [\leq''], [\leq'''])&\mapsto [\iota_{Q_b(k-1), Q( n_a, 0, n_c), Q_{b, k}(n_c)}(\leq'', \leq''', \leq')]\\
 g_{c, l}: \qh(Q_{b, l}(n_b)) \times \qh(Q_b(l-1))\times& \qh(Q(n_a, 0, n_c) \rightarrow  \qh_{b, l}(Q(n_a, n_b, n_c))\\
         ([\leq'], [\leq''], [\leq'''])&\mapsto [\iota_{Q_c(l-1), Q( n_a, n_b, 0), Q_{c, l}(n_c)}(\leq'', \leq''', \leq')]\\
    \end{align*}
\end{proposition}
\begin{proof}
   By Remark \ref{remark_partial_order_restriction} and Remark \ref{remark_adapted_restriction}, $f_b, f_c, f_{b, k}, f_{c, l}$ and $ g_b, g_c, g_{b, k}, g_{c, l}$ are well-defined. Moreover, it is clear that $g_{x, y}\circ f_{x, y}$ is the identity for $x\in \{b,c\}$ and $y\in \{1, \dots, n_b+1\}$ respectively $y\in \{1, \dots, n_c+1\}$. It remains to be shown that $f_{x, y}\circ g_{x, y}$  is the identity.
   We show this in the case $x=b$ and $1\leq y\leq n_b$, the case $y=n_b+1$ is similar, and the case $x=c$ follows by symmetry. Note that for the case $y=1$ we use the convention that $\qh(Q_b(0))$ contains the empty order.\\
   Let $[\leq]\in \qh_{b, k}(Q(n_a, n_b, n_c))$.
   We have to show that the adapted partial order $\hat{\leq}$, obtained by  restricting $\leq$ to three partial orders $\leq':=\leq_{|Q(n_a, 0, n_c)}$,  $\leq'':=\leq_{|Q_b(k-1)}$ and   $\leq''':=\leq_{|Q_{b, k}(n_b)}$ and then applying $\iota_{Q_{b, k}(n_b), Q(n_c, 0, n_c), Q_b(k-1)}$, is equivalent to $\leq$, that is, we have to show that $\leq$ and $\hat{\leq}$ give rise to the same costandard modules.\\
   For any vertex $$x\in \{a_i|0\leq i\leq n_a\}\cup \{b_i|1\leq i\leq n_b\}\cup \{c_i|1\leq i\leq n_c\}$$ of $Q(n_a, n_b, n_c)$, denote by $\nabla_{x}^{Q(n_a, n_b, n_c)}$ the costandard module for $Q(n_a, n_b, n_c)$ at $x$ with respect to $\leq$ and by  $\hat{\nabla}_{x}^{Q(n_a, n_b, n_c)}$ the costandard module for $Q(n_a, n_b, n_c)$ at $x$ with respect to $\hat{\leq}$.\\
   Similarly, for any $x\in \{a_i|0\leq i\leq n_a\}\cup \{c_i|1\leq i\leq n_c\}$ denote by $\nabla_{x}'$ the standard module for $Q(n_a, 0, n_c)$ at $x$ with respect to $\leq'$, for any $x\in \{b_i|1\leq  i\leq k-1\}$ denote by $\nabla_x''$ the standard module for $Q_b(k-1)$ at $x$ with respect to $\leq''$ and for any $x\in \{b_i|k\leq i\leq n_b\}$ denote by $\nabla_x'''$ the standard module for $Q_{b, k}(n_b)$ at $x$ with respect to $\leq'''$.\\
   Moreover, let $e':=\sum_{i=1}^{n_b}e_{b_i}$, $e'':=\sum_{i=1}^{n_a}e_{a_i}+\sum_{i=1}^{n_c}e_{c_i}+\sum_{i=k}^{n_b}e_{b_i}$ and \\$e''':=\sum_{i=1}^{n_a}e_{a_i}+\sum_{i=1}^{n_c}e_{c_i}+\sum_{i=1}^{k-1}e_{b_i}$ and let
   \begin{align*}
       &\pi': \field Q(n_a, n_b, n_c)\rightarrow \field Q(n_a, 0, n_c)\cong  \field Q(n_a, n_b, n_c)/ (\field Q(n_a, n_b, n_c)e' \field Q(n_a, n_b, n_c)),\\
       &\pi'': \field Q(n_a, n_b, n_c)\rightarrow \field Q_b(k-1)\cong  \field Q(n_a, n_b, n_c)/ (\field Q(n_a, n_b, n_c)e'' \field Q(n_a, n_b, n_c)),
   \end{align*}
   and
    \begin{align*}
       \pi''': \field Q(n_a, n_b, n_c)\rightarrow \field Q_{b, k}(n_b)\cong  \field Q(n_a, n_b, n_c)/ (\field Q(n_a, n_b, n_c)e''' \field Q(n_a, n_b, n_c)),
   \end{align*}
   be the canonical projections.\\
   Recall that $\leq'$, $\leq''$ and $\leq'''$ are the restrictions of both $\leq$ and $\hat{\leq}$ to  $Q(n_a, 0, n_c)$, respectively  $Q_b(k-1)$, respectively  $Q_{b, k}(n_b)$.\\
    Let $x\in \{a_i|0\leq i\leq n_a\}\cup \{c_i|1\leq i\leq n_c\}$. Then, since
    there are no arrows in $Q(n_a, n_b, n_c)$ from any vertex in the vertex sets of $Q_b(k-1)$ or $Q_{b, k}(n_b)$ to any vertex in the vertex set of $Q(n_a, 0, n_c)$, the injective at $x$ for  $Q(n_a, n_b , n_c)$ coincides with the restriction along $\pi'$ of the injective at $x$ for $Q(n_a, 0, n_c)$. Since $\leq'$ is the restriction of both $\leq$ and $\hat{\leq}$ to  $Q(n_a, 0, n_c)$, this implies that 
    \begin{align*}
        \nabla_x\cong \Res_{\pi'}\nabla_x'\cong \hat{\nabla}_x.
    \end{align*}
    Now let $x=b_i$ for $1\leq i<k$.
    Then by definition of  $\iota_{Q_{b, k}(n_b), Q(n_c, 0, n_c), Q_b(k-1)}$, $b_i\hat{\leq}a_j$ for all $0\leq j\leq n_a$ and $b_i\hat{<}b_j$ for all $k\leq j\leq n_b$. Hence 
    \begin{align*}
      \hat{\nabla}_x\cong \Res_{\pi''}\nabla_x''.
    \end{align*}
    On the other hand, $x=b_i<a_0$ by definition of $\qh_{b, k}(Q(n_a, n_b, n_c)$. Since any path from any vertex in $\{a_j|0\leq j\leq n_a\}\cup \{c_j|1\leq j\leq n_c\}\cup\{b_j|k\leq j\leq n_b\}$ to $x=b_i$ must pass through $a_0$, we therefore also have 
    \begin{align*}
        \nabla_x\cong \Res_{\pi''}\nabla_x''.
    \end{align*}
   Finally, let  $x=b_i$ for $k\leq i\leq n_b$. Suppose first that $\nabla'''_x$ is not an injective $\field Q_{b, k}(n_b)$-module. Then, there is some $k\leq j< i$ such that $b_j>'''b_i$. Since $\leq'''$ is the restriction of both $\leq$ and $\hat{\leq}$ to  $Q_{b, k}(n_b)$, this implies that $b_j>b_i$ and $b_j\hat{>}b_i$, and since any path from any vertex in $\{a_l|0\leq l\leq n_a\}\cup \{c_l|1\leq l\leq n_c\}\cup\{b_l|1\leq j<k\}$ to $x=b_i$ must pass through $b_j$, we have
     \begin{align*}
        \nabla_x\cong \Res_{\pi'''}\nabla_x'''\cong \hat{\nabla}_x.
    \end{align*}
    Suppose on the other hand that $\nabla'''_x$ is an injective $\field Q_{b, k}(n_b)$-module. Then $x=b_i\geq''' b_j$ for all $k\leq j\leq i$, so that $x\geq b_j$ for all $k\leq j\leq i$. In particular, $x\geq b_k$ and $b_k>a_j$ for every $1\leq j\leq n_a$. Moreover, $b_k>a_0>b_j$ for all $1\leq j<k$. Since the vertex set $\{a_j|0\leq j\leq n_a\} \cup\{b_j|1\leq j\leq i\}$ accounts for all composition factors of the injective $\field Q(n_a, n_b, n_c)$-module at $x$, this implies that $\nabla_x$ is injective.\\ 
    Similarly, by definition of $\hat{\leq}$, $y\hat{<}b_i$ for any vertex $y\in \{a_l|0\leq l\leq n_a\}\cup \{c_l|1\leq l\leq n_c\}\cup\{b_l|1\leq j<k\}$, and since $x=b_i\geq''' b_j$ for all $k\leq j\leq i$, $x\hat{\geq} b_j$ for all $k\leq j\leq i$, so that $\hat{\nabla}_x$ is also injective. 
\end{proof}
\begin{remark}\label{last_remark}
    Note that using the convention that the sets $\qh(Q_{b, n_b+1}(n_b))$ and $\qh(Q_{c, n_c+1}(n_c))$ contain the empty order, the statements for $f_b$ and $g_b$ and $f_c$ and $g_c$ become the case $k=n_b+1$ respectively $l=n_c+1$ in the corresponding statement for $f_{b, k}$ and $g_{b, k}$, respectively $f_{c, l}$ and $g_{c, l}$.
\end{remark}
\begin{theorem}\label{thm_counting}
    The quasi-hereditary structures on $\field Q(n_a, n_b, n_c)$ admitting a regular exact Borel subalgebra are counted by 
    \begin{align*}
        C_{n_b+1} C_{n_a+n_c+1}+C_{n_c+1} C_{n_a+n_b+1}-C_{n_a+1}C_{n_b+1}C_{n_c+1}
    \end{align*}
    The quasi-hereditary structures on $\field Q(n_a, n_b, n_c)^{\op}$ admitting a regular exact Borel subalgebra are counted by 
    \begin{align*}
        C_{n_b} C_{n_a+n_c+1}+C_{n_c} C_{n_a+n_b+1}-C_{n_b} C_{n_c} C_{n_a+1}.
    \end{align*}
\end{theorem}
\begin{proof}
    By Corollary \ref{corollary_special_quiver_regular}, the set $\qh^B(Q(n_a, n_b, n_c))$ is the (in general not disjoint) union of the sets $\qh_{b, k}(Q(n_a, n_b, n_c))$ for $1\leq k\leq  n_b+1$ and $\qh_{c, l}(Q(n_a, n_b, n_c))$ for $0\leq l\leq n_c+1$. Note that by definition of $\qh_{b, k}(Q(n_a, n_b, n_c))$, the sets $\qh_{b, k}(Q(n_a, n_b, n_c))$, $1\leq k\leq n_b+1$  are pairwise disjoint, and analogously, the sets  $\qh_{c, l}(Q(n_a, n_b, n_c))$, $0\leq l\leq n_c+1$ are pairwise disjoint.
    Hence
    \begin{align*}
        |\qh^B(Q(n_a, n_b, n_c))|=&\sum_{k=0}^{n_b+1}|\qh_{b, k}(Q(n_a, n_b, n_c))|+\sum_{l=0}^{n_c+1}|\qh_{c, l}(Q(n_a, n_b, n_c))|\\
        &- \sum_{k=0}^{n_b+1}\sum_{l=0}^{n_c+1}|\qh_{b, k}(Q(n_a, n_b, n_c))\cap \qh_{c, l}(Q(n_a, n_b, n_c)|
    \end{align*}
    Moreover, by Proposition \ref{proposition_bijections} and Remark \ref{last_remark}, we have for all $1\leq k\leq n_b+1$ 
    \begin{align*}
        |\qh_{b, k}(Q(n_a, n_b, n_c))|= |\qh(Q_{b, k}(n_b))| \cdot|\qh(Q_b(k-1))|\cdot | \qh(Q(n_a, 0, n_c)|.
    \end{align*}
    Since $Q_{b, k}(n_b)$, $Q_b(k-1)$ and $Q(n_a, 0, n_c)$ are quivers of directed type $\textup{A}_n$, we can use  \cite[Corollary 4.8]{combinatorics},  to conclude that
    \begin{align*}
       |\qh_{b, k}(Q(n_a, n_b, n_c))|=C_{n_b-k+1}C_{k-1}C_{n_a+n_c+1},  
    \end{align*}
    where, for the case $k=1$ and $k=n_b+1$ we use that by our convention there is exactly one quasi-hereditary structure on the quiver of type $\textup{A}_0$.
    Similarly
     \begin{align*}
       |\qh_{c, l}(Q(n_a, n_b, n_c))|=C_{n_c-l+1}C_{l-1}C_{n_a+n_b+1} 
    \end{align*}
    for all $1\leq l\leq n_c+1$
    Moreover, using the bijection in Proposition \ref{proposition_bijections}, we see that given $1\leq k\leq n_b+1$ and $1\leq l\leq n_c+1$, 
    the intersection $$\qh_{b, k}(Q(n_a, n_b, n_c))\cap \qh_{c, l}(Q(n_a, n_b, n_c))$$ is in bijection with $$\qh(Q_{b, k}(n_b))\times\qh(Q_b(k-1))\times\qh_{c, l}(Q(n_a, 0, n_c)),$$ which in turn is in bijection with  $$\qh(Q_{b, k}(n_b))\times\qh(Q_b(k-1))\times \qh(Q_{c, l}(n_c))\times \qh(Q_c(l-1))\times \qh(Q(n_a, 0, 0)).$$ 
    Hence, by \cite[Corollary 4.8]{combinatorics},
    \begin{align*}
     &|\qh_{b, k}(Q(n_a, n_b, n_c))\cap \qh_{c, l}(Q(n_a, n_b, n_c))|\\
     =& |\qh(Q_{b, k}(n_b))|\cdot|\qh(Q_b(k-1)|\cdot | \qh(Q_{c, l}(n_c)|\cdot| \qh(Q_c(l-1))|\cdot| \qh(Q(n_a, 0, 0)|\\
     =& C_{n_b-k+1} C_{k-1} C_{n_c-l+1} C_{l-1} C_{n_a+1}.
     \end{align*}
     Thus, 
     \begin{align*}
         |\qh^B(Q(n_a, n_b, n_c)|
         =&\sum_{k=1}^{n_b+1}C_{n_b-k+1}C_{k-1}C_{n_a+n_c+1}\\&+\sum_{l=1}^{n_c+1}C_{n_c-l+1}C_{l-1}C_{n_a+n_b+1}\\ &-\sum_{k=1}^{n_b+1}\sum_{l=1}^{n_c+1}C_{n_b-k+1}C_{k-1}C_{n_c-l+1}C_{l-1}C_{n_a+1}\\
         =&C_{n_b+1}C_{n_a+n_c+1}+C_{n_c+1}C_{n_a+n_b+1}-C_{n_a+1}C_{n_b+1}C_{n_c+1},
     \end{align*}
     where the last equality follows from the recursive formula for Catalan numbers.\\
    Similarly, by \ref{corollary_special_quiver_regular}, the set $\qh^B(Q(n_a, n_b, n_c)^{\op})$ is the (in general not disjoint) union of the sets $\qh_{b}(Q(n_a, n_b, n_c))$ and $\qh_{c}(Q(n_a, n_b, n_c))$.
    Hence
    \begin{align*}
        |\qh^B(Q(n_a, n_b, n_c)^{\op})|=&|\qh_{b}(Q(n_a, n_b, n_c))|+|\qh_{c}(Q(n_a, n_b, n_c))|\\ &- |\qh_{b}(Q(n_a, n_b, n_c))\cap \qh_{c}(Q(n_a, n_b, n_c)|.
    \end{align*}
    Moreover, by Proposition \ref{proposition_bijections}, we have
    \begin{align*}
        |\qh_{b}(Q(n_a, n_b, n_c))|= |\qh(Q_{b}(n_b))| \cdot| \qh(Q(n_a, 0, n_c)|=C_{n_b}C_{n_a+n_c+1}
    \end{align*}
    and
     \begin{align*}
       |\qh_{c}(Q(n_a, n_b, n_c))|=C_{n_c}C_{n_a+n_b+1} .
    \end{align*}
    Moreover, as before, we have that the intersection
     $$\qh_{b}(Q(n_a, n_b, n_c))\cap \qh_{c}(Q(n_a, n_b, n_c))$$ is in bijection with $$\qh(Q_{b}(n_b))\times \qh_{c}(Q(n_a, 0, n_c)),$$ which in turn is in bijection with  $$\qh(Q_{b}(n_b))\times\qh(Q_{c}(n_c))\times \qh(Q(n_a, 0, 0)).$$ 
    Hence, by \cite[Corollary 4.8]{combinatorics},
    \begin{align*}
     &|\qh_{b}(Q(n_a, n_b, n_c))\cap \qh_{c}(Q(n_a, n_b, n_c))|\\
     =& |\qh(Q_{b}(n_b))|\cdot | \qh(Q_{c}(n_c)|\cdot| \qh(Q(n_a, 0, 0)|\\
     =& C_{n_b} C_{n_c} C_{n_a+1}.
     \end{align*}
     Thus, 
     \begin{equation*}
         |\qh^B(Q(n_a, n_b, n_c)|^{\op}
         =C_{n_b}C_{n_a+n_c+1}-C_{n_c}C_{n_a+n_b+1}-C_{n_b}C_{n_c}C_{n_a+1}.\qedhere
     \end{equation*}
\end{proof}
For type $\textup{D}$, we thus obtain the following table, taking the second to last column from \cite[p. 22-23]{combinatorics}:
\begin{center}\label{tableD}
\begin{tabular}{c|c| c| c }
    $Q$ & $|\qh^B(Q)|$  &  $|\qh(Q)|=|\qh(Q^{\op})|$ & $\lim_{n\rightarrow\infty} \frac{|\qh^B(Q)|}{|\qh(Q)|}$ \\
    \hline
    $Q(n, 1, 1)$ & $4(C_{n+2}-C_{n+1})$ & \multirow{2}{*}{$2C_{n+3}-3C_{n+2}$} &  3/5 \\
    $Q(n, 1, 1)^{\op}$ & $2C_{n+2}-C_{n+1}$ &  & 7/20  \\
    \hline
    $Q(1, n, 1)$ & $2C_{n+2}+C_{n+1}$ & \multirow{2}{*}{$3C_{n+2}-C_{n+1}$} & 9/11 \\
    $Q(1, n, 1)^{\op}$ & $C_{n+2}+3C_n$ &  & 19/44
\end{tabular}
    
\end{center}

In the following table, we have calculated some numbers for small $n$:
\begin{center}
\begin{tabular}{c|c| c| c}
    $Q$ & $|\qh^B(Q)|$ &  $|\qh^B(Q^{\op})|$ &  $|\qh(Q)|=|\qh(Q^{\op})|$\\
    \hline
    $Q(1, 1, 1)$ & 12 & 8 &  13 \\
    $Q(2, 1, 1)$ & 36 & 23 & 42 \\
    $Q(3, 1, 1)$ & 112 & 70 & 138 \\
    \hline
    $Q(1, 1, 1)$ & 12 & 8 &  13 \\
    $Q(1, 2, 1)$ & 33 & 20 & 37\\
    $Q(1, 3, 1)$ & 98 & 57 & 112\\
\end{tabular}
    
\end{center}
As mentioned in \cite[p.23]{combinatorics}, the sequence $(2C_{n+3}-3C_{n+2})_n$ is listed in the OEIS as \href{https://oeis.org/A070031}{A070031(n+1)}. Moreover, the sequences $(2C_{n+2}-C_{n+1})_n$ and $(2C_{n+2}+C_{n+1})_n$ can be found in the OEIS as \href{https://oeis.org/A000782}{A000782(n+2)} and \href{https://oeis.org/A376161}{A376161(n+1)} respectively. At the point of publication, the remaining sequences are not listed in the OEIS.\\
In a similar way to type $\textup{D}$, we can also count the quasi-hereditary structures which admit regular exact Borel subalgebras on type $\textup{E}$; again, we have taken the last column directly from \cite[p.23]{combinatorics}.
\begin{center}
\begin{tabular}{c | c|c|c|c}
   Type & $Q$ & $|\qh^B(Q)|$ &  $|\qh^B(Q^{\op})|$ &  $|\qh(Q)|=|\qh(Q^{\op})|$\\
   \hline
   \multirow{2}{*} {$\textup{E}_6$} &  $Q(1,2,2)$   & 90 & 48 & 106\\
    & $Q(2,2,1)$  & 104 & 60 & 130 \\
    \hline
   \multirow{3}{*} {$\textup{E}_7$} & $Q(1,3,2)$  & 266 & 134 & 322\\
   & $Q(2,3,1)$  & 320 & 177 & 416\\
   & $Q(3,2,1)$  & 334 & 188 & 453 \\
   \hline
   \multirow{3}{*} {$\textup{E}_8$} & $Q(1,4,2)$  & 828 & 404 & 1020\\
    & $Q(2,4,1)$  & 1026 & 555 & 1368\\
    & $Q(4,2,1)$  & 1098 & 609 & 1584\\
\end{tabular}
\end{center}

\section{Acknowledgements}

I would like to thank Teresa Conde and Steffen König for discussions resulting in Theorem \ref{thm_reedy}, as well as Julian Külshammer for many other helpful discussions. Moreover, I would like to thank Georgios Dalezios, Teresa Conde, Steffen König and Julian Külshammer for sharing their results and explaining them to me.

\bibliography{monomial}
\end{document}